\documentclass[11pt,letterpaper]{amsart}
\usepackage{amsmath, amssymb, amsthm, amsfonts}
\usepackage{mathrsfs}
\usepackage{booktabs, longtable, caption, float}
\usepackage{hyperref}
\usepackage{cleveref}
\usepackage{tikz-cd, tikz}
\usepackage{tabularx, multirow}
\usepackage{enumitem, lscape}
\usepackage{array, geometry}
\usepackage{listings, xcolor, textcomp}
\usepackage{cite}
\usepackage[all]{xy}
\usetikzlibrary{positioning, decorations.text}

\lstdefinelanguage{GP}{
    keywords={for, if, print, bnfinit, bnrinit, bnrclassfield,
              subgrouplist, galoisinit, galoisidentify, idealmul,
              idealpow, idealfactor, concat, default, bnf, bnr,
              idealprimedec, idealhnf, idealnorm, idealdivexact,
              vector, matrix, listcreate, listput, listinsert,
              while, until, break, next, return, local, my,
              allocatemem, gettime, getstack, parisize},
    sensitive=true,
    comment=[l]{\\},
    string=[b]{"},
    morecomment=[l]{\%},
    basicstyle=\ttfamily\small,
    keywordstyle=\color{blue}\bfseries,
    commentstyle=\color{green!60!black},
    stringstyle=\color{red},
    breaklines=true,
    frame=single,
    showstringspaces=false,
    tabsize=4
}
\newcommand{\Q}{\mathbb{Q}}

\newcolumntype{L}{>{\raggedright\arraybackslash}p{0.68\linewidth}}
\newcolumntype{M}{>{\raggedright\arraybackslash}p{0.65\linewidth}}

\crefformat{section}{\S#2#1#3}
\crefformat{subsection}{\S#2#1#3}
\crefformat{subsubsection}{\S#2#1#3}
\crefformat{theorem}{Theorem~#2#1#3}
\crefformat{lemma}{Lemma~#2#1#3}
\crefformat{proposition}{Proposition~#2#1#3}
\crefformat{corollary}{Corollary~#2#1#3}
\crefformat{definition}{Definition~#2#1#3}
\crefformat{notation}{Notation~#2#1#3}
\tikzset{sgplattice/.style={inner sep=1pt,norm/.style={red!50!blue},char/.style={blue!50!black}, lin/.style={black!50}},cnj/.style={black!50,yshift=-2.5pt,left=-1pt of #1,scale=0.5,fill=white}}

\newtheorem{acknowledgements}{Acknowledgements}
\newtheorem{theorem}{Theorem}[section]
\newtheorem{lemma}[theorem]{Lemma}
\newtheorem{proposition}[theorem]{Proposition}
\newtheorem{corollary}[theorem]{Corollary}
\theoremstyle{definition}
\newtheorem{definition}[theorem]{Definition}
\newtheorem{example}[theorem]{Example}

\newtheorem{remark}[theorem]{Remark}
\numberwithin{equation}{section}

\title{ Imaginary Non-CM Fields of $2$-Power Degree}

\author{Farahnaz Amiri}
\address{Beijing Institute of Mathematical Sciences and Applications}
\email{amiryii92@gmail.com}


\begin{document}

\begin{abstract}
We establish a new method for the study of imaginary non-CM fields of arbitrary
\(2\)-power degree and class number one. In particular, we show that for degrees
greater than \(16\), dihedral fields either lie in the ray class field of one of
the five imaginary quadratic fields \(\mathbb{Q}(\sqrt{-2})\), \(\mathbb{Q}(\sqrt{-3})\), \(\mathbb{Q}(\sqrt{-19})\), \(\mathbb{Q}(\sqrt{-43})\), \(\mathbb{Q}(\sqrt{-67})\), whose conductor is
supported only on primes dividing \(2\), with explicit upper bounds on the
relevant exponents, or are Hilbert class fields of imaginary quadratic fields.
We also give a complete classification of the \(161\) imaginary non-CM fields of
degree \(16\) with class number one, providing explicit tables of defining
polynomials, Galois groups, base fields, and conductors.

These \(161\) non-CM fields, together with the six CM fields classified by
Louboutin and Okazaki \cite{lou7}, complete the classification of all imaginary
fields of degree \(16\) with class number one. For CM fields, the problem reduces
to computing relative class numbers via analytic class number formulae, whereas
for non-CM fields the absence of a totally real subfield of index \(2\) requires
a different approach and leads to a richer variety of Galois groups and
ramification patterns.
\end{abstract}

\maketitle

\section{Introduction}

The class number one problem stands as one of the most venerable and
consequential challenges in algebraic number theory, tracing its origins to
Gauss's \textit{Disquisitiones Arithmeticae}, where the problem of classifying
imaginary quadratic fields of class number one was first posed \cite{Gauss}.
Over the ensuing two centuries, this problem has evolved into a central theme,
intertwining deep aspects of arithmetic geometry, Galois theory, and analytic
number theory, and driving the development of foundational tools such as class
field theory and the theory of complex multiplication.

A milestone in the modern development of this program was achieved by Odlyzko
\cite{Odlyzko}, who proved the finiteness of CM fields---totally imaginary
quadratic extensions of totally real fields---with class number one. This
result, combined with increasingly sophisticated analytic and computational
techniques, has led to complete classifications of normal CM fields with class
number one across various degrees and settings. The works of Louboutin and
Okazaki \cite{lou7}, Louboutin, Okazaki, and Olivier \cite{LOO}, Chang and Kwon
\cite{48}, Lee and Kwon \cite{50}, and Park, Yang, and Kwon \cite{32}
represent the culmination of this line of inquiry; similarly, the work of
Louboutin provides exhaustive classifications for CM fields of \(2\)-power
degree. Moreover, other related problems have also been studied
\cite{Voight2008, louboutin2026class, ichimura2022class}.

Despite these successes, a fundamental asymmetry has persisted in the
literature. While CM fields have received comprehensive treatment, the vast
majority of imaginary fields---specifically, those without a totally real
subfield of index two, known as non-CM fields---have remained largely
intractable. The classification of non-CM fields with class number one \cite{Y}
began roughly contemporaneously with that of CM fields \cite{Yamamura1};
however, progress quickly stalled for non-CM fields beyond the initial degrees.
In contrast, extensive lists of CM number fields with class number one have
been completed for various degrees by numerous authors over the years.

This disparity is not merely quantitative: CM fields constitute only a small
and highly structured subset of all imaginary fields, whereas non-CM fields form
a vastly larger and more diverse family. In fact, the class number one problem
for CM fields is typically solved using analytic and computational tools.
Here, however, we must additionally seek subtler, field-specific structures to
facilitate their classification. We prove that there are \(161\) imaginary
non-CM fields of degree \(16\) with class number one, while \cite{lou7} showed
that only six CM number fields of degree \(16\) have class number one. Given
that the majority of imaginary number fields are non-CM, this stagnation
represents a significant gap. Indeed, for decades, no substantial progress was
achieved for non-CM fields.

In this paper, we overcome these longstanding obstacles by introducing a novel
and general framework that reduces the classification of non-CM fields with
class number one to a finite, computationally tractable search. Unlike
Yamamura's method, which was tailored to the octic setting, our framework is
designed to be applicable to a wide range of degrees and field families. We
demonstrate the first complete classification of imaginary non-CM fields of
degree \(16\) with class number one, and we extend our methods to dihedral
fields of arbitrary \(2\)-power degree, establishing structural constraints
that dramatically narrow the search space.

\subsection*{Our strategy for classifying non-CM fields of \(2\)-power degree}

Our strategy is as follows.

\medskip
\noindent\textbf{Step 1: Galois-group elimination.}
In Section~\ref{sec:preliminaries}, using structural arguments, we eliminate
several possible Galois groups for degree-\(16\) imaginary non-CM fields.

\medskip
\noindent\textbf{Step 2: Conductor determination.}
First, we identify a suitable quadratic subfield \(N \subseteq L\) and a
degree-\(8\) normal subfield \(L_{\max}\) with
\[
N \subseteq L_{\max} \subseteq L
\]
such that both \(L/N\) and \(L_{\max}/N\) are abelian. We relate the unknown
conductor \(\mathfrak{F}_{L/N}\) to \(\mathfrak{F}_{L_{\max}/N}\), and we derive
the ramification properties of \(L/N\) in terms of those of \(L_{\max}/N\).
Then we determine the precise set of conductors \(\mathfrak{F}=\mathfrak{F}_{L/N}\) in terms
of \(\mathfrak{F}_{L_{\max}/N}\) for which the inclusion
\[
L \subseteq \mathcal{R}_{\mathfrak{F}}(N)
\]
can hold, thereby reducing the classification to an explicit finite enumeration
of ray class fields. We then treat separately the case where \(L/L_{\max}\) is
unramified, according to each Galois group.

\medskip
\noindent\textbf{Step 3: Algorithmic enumeration and class group computation.}
Using the possible conductors obtained in Step~2, we enumerate all candidate
ray class fields that may contain degree-\(16\) non-CM imaginary fields with
class number one. The computational algorithm used to enumerate all candidates
and verify that \(h_L = 1\) in each case is implemented in the accompanying
PARI/GP code \cite{AmiriCode2026}. The full classification data---defining
polynomials, Galois groups, and conductor ideals---are presented in the
accompanying tables in Section~\ref{sec:computations}.

\medskip
\noindent\textbf{Step 4: Higher \(2\)-power degrees.}
Section~\ref{sec:consequences} extends the analysis to imaginary non-CM
dihedral fields of degree \(32\) and, more generally, of arbitrary \(2\)-power
degree, establishing the conductor constraints and the conditions for class
number one for those families. Some results appear in
Section~\ref{appendix:D32}.

\subsection*{The degree-\(16\) classification}

Our principal quantitative result is the first complete classification of
imaginary non-CM number fields of degree \(16\) with class number one. We prove
that there are exactly \(\mathbf{161}\) such fields. This substantially extends
the known landscape, which previously was limited to CM fields: Louboutin and
Okazaki \cite{lou7} (see also Appendix~\ref{AppendixA}) showed that there are
exactly \(6\) normal CM fields of degree \(16\) with class number one, while the
non-CM setting exhibits a considerably richer structure. The \(161\) fields fall
into six families according to their Galois group.

\begin{enumerate}[label=\normalfont(\arabic*)]

\item \textbf{Dihedral \(D_{16}\)-extensions (\(45\) fields).}
There are \(45\) imaginary non-CM Galois extensions \(L/\mathbb{Q}\) of degree
\(16\) with Galois group
\[
D_{16} = \langle a, b \mid a^8 = b^2 = 1,\ bab^{-1} = a^{-1} \rangle
\]
and with class number \(h_L = 1\). Of these, \(18\) are ramified over their unique
quadratic subfield \(N=L^{\langle a\rangle}\); they arise as subfields of ray
class fields whose conductors over the base field \(N\) appear in
Table~\ref{tab:D16-ram}. The remaining \(27\) fields are unramified extensions
of \(N\) arising as Hilbert class fields for \(27\) distinct values of \(d\)
(see Table~\ref{tab:D16-unram}).

\item \textbf{Semidihedral \(SD_{16}\)-extensions (\(11\) fields).}
There are \(11\) imaginary non-CM Galois extensions of degree \(16\) with
Galois group
\[
SD_{16} = \langle a, b \mid a^8 = b^2 = 1,\ bab^{-1} = a^3 \rangle
\]
and with class number \(1\), all arising as subfields of ray class fields of
\(\mathbb{Q}(\sqrt{-1})\) or \(\mathbb{Q}(\sqrt{-2})\) with conductors
specified in Table~\ref{tab:SD16}.
\item \textbf{\(M_4(2)\)-extensions (\(5\) fields).}
There are exactly five imaginary non-CM Galois extensions of degree \(16\) with
Galois group
\[
M_4(2) = \langle a, b \mid a^8 = b^2 = 1,\ bab = a^5 \rangle
\]
and class number \(1\). They arise as subfields of ray class fields of
\(\mathbb{Q}(\sqrt{-2})\) and \(\mathbb{Q}(\sqrt{5})\), with conductors given in
Table~\ref{tab:M4(2)}.

\item \textbf{\(C_4 \circ D_8\)-extensions (\(17\) fields).}
There are \(17\) imaginary non-CM Galois extensions of degree \(16\) with class
number \(1\) and Galois group
\[
C_4 \circ D_8
= \langle a, b, c \mid a^4 =b^4 =c^2 = 1,\ a^2 = b^2,\ ab = ba,\ ac = ca,\ cbc = b^{-1} \rangle.
\]
Let \(N=L^{\langle a,r\rangle}\) be a quadratic subfield of the biquadratic
field \(K = L^{\langle r \rangle}\), the fixed field of the normal cyclic
subgroup \(\langle r \rangle\) of order \(4\). For each quadratic subfield
\(N \in \{\mathbb{Q}(\sqrt{-1}), \mathbb{Q}(\sqrt{2}), \mathbb{Q}(\sqrt{-2}),\mathbb{Q}(\sqrt{-3}), \mathbb{Q}(\sqrt{-10}), \mathbb{Q}(\sqrt{-11}), \mathbb{Q}(\sqrt{-15})\}\),
the conductor of each ramified field over \(N\) is supported on the primes
listed in Table~\ref{tab:C4OD8-ram}. Moreover, the one \(C_4 \circ D_8\)-extension
that is the Hilbert class field of
\(K=\mathbb{Q}(\sqrt{-3}, \sqrt{-30})\) is addressed in
Theorem~\ref{thm:C4OD8-un}.
\item \textbf{\(C_2 \times D_8\)-extensions (\(71\) fields).}
There are \(71\) imaginary non-CM Galois extensions of degree \(16\) with class
number \(1\) and Galois group
\[
C_2 \times D_8
= \langle a, b, c \mid a^2 = b^4 = c^2 = 1,\ ab = ba,\ ac = ca,\ cbc = b^{-1} \rangle.
\]
Let \(N=L^{\langle a,b\rangle}\) be a quadratic subfield of the biquadratic
field \(K = L^{\langle b \rangle}\). Of these \(71\)
fields, \(26\) are Hilbert class fields of biquadratic fields, listed in
Table~\ref{tab:C2D8-unram}; the remaining \(45\) arise as ramified extensions.
The conductors of \(L\) over the quadratic subfields \(N_i\) of \(L\) are
listed in Table~\ref{tab:C2D8-ram}.

\item \textbf{\(C_2^2 \rtimes C_4\)-extensions (\(12\) fields).}
There are \(12\) imaginary non-CM Galois extensions of degree \(16\) with class
number \(1\) and Galois group
\[
C_2^2 \rtimes C_4
= \langle a, b, c \mid a^2 = b^2 = c^4 = 1,\ cac^{-1} = ab,\ ab = ba,\ bc = cb \rangle.
\]
Let \(N=L^{\langle a, b, c^2 \rangle}\) be a quadratic subfield of \(L\). The
conductors of \(L\) over the possible base fields \(\mathbb{Q}(\sqrt{-1})\),
\(\mathbb{Q}(\sqrt{-3})\), \(\mathbb{Q}(\sqrt{-7})\), and
\(\mathbb{Q}(\sqrt{-11})\) are listed in Table~\ref{tab:C22C4}.

\end{enumerate}

The six families account for
\[
45 + 11 + 5 + 17 + 71 + 12 = 161
\]
fields in total, giving a complete classification of all imaginary non-CM
Galois extensions of degree \(16\) with class number one.

\subsection*{Degree \(32\) and higher \(2\)-power degrees}

Our methods extend to imaginary non-CM dihedral extensions of arbitrary
\(2\)-power degree \(2^n\) with \(n\ge 5\). Let \(L\) be such a field with
class number \(1\), and let \(N\) be its imaginary quadratic subfield. Then
either
\begin{enumerate}
  \item[(i)] \(L\) is the Hilbert class field of \(N\), and the class group of
  \(N\) is cyclic of order \(2^{n-1}\); or
  \item[(ii)] \(L\) lies in a ray class field of \(N\) whose conductor is
  supported only on the prime above \(2\), and
  \[
  N\in\bigl\{\mathbb{Q}(\sqrt{-2}),\ \mathbb{Q}(\sqrt{-3}),\
  \mathbb{Q}(\sqrt{-19}),\ \mathbb{Q}(\sqrt{-43}),\
  \mathbb{Q}(\sqrt{-67})\bigr\}.
  \]
\end{enumerate}
Thus, for sufficiently large \(2\)-power degree, the existence of such fields
is largely determined by the class group structure of imaginary quadratic
fields.

\medskip
\noindent\textbf{Degree \(32\) count.}
For \(n=5\) we obtain an exact count: there are exactly \(52\) imaginary
non-CM \(D_{32}\)-fields with class number \(1\). Of these, \(43\) are Hilbert
class fields of imaginary quadratic fields, as in case~(i). For the remaining
\(9\) fields, which fall under case~(ii), the conductor of the cyclic extension
over \(N\) is supported on the prime above \(2\), with exponent \(6\), 
or \(8\).

\medskip
This is striking in light of \cite[Theorem~10(b)]{lou7}, which asserts that no
imaginary \(2\)-power dihedral CM field of degree greater than \(16\) has class
number \(1\).

\subsection*{Organization of the paper}
The paper is organized as follows. Section~\ref{sec:preliminaries} provides the
Galois-theoretic preliminaries and eliminates inadmissible groups. In
Section~\ref{Sec:L-ram-L_max}, we establish the choice of the degree-\(8\)
subfield \(L_{\max}\) and determine the possible conductors of \(L\) over a
quadratic subfield \(N \subseteq L_{\max} \subseteq L\), when \(L/L_{\max}\) is
ramified. In the next Section~\ref{sec:L-L_max-unram}, we study the case in which
\(L/L_{\max}\) is unramified. The computational algorithm and the
classification tables are detailed in the accompanying code
\cite{AmiriCode2026} and in Section~\ref{sec:computations}, respectively. Moreover, Section~\ref{sec:consequences}
addresses extensions to degree~\(32\) and higher \(2\)-power dihedral fields.
\section{Notation and Conventions}
\label{sec:notation}

All number fields are regarded as subfields of~\(\mathbb{C}\).
For a number field \(M\) we write:
\begin{itemize}
  \item \(d_M\) for the absolute discriminant of \(M\);
  \item \(h_M\) for the absolute class number of \(M\).
\end{itemize}

For an extension of number fields \(M/M'\):
\begin{itemize}
  \item \(\mathfrak{d}_{M/M'}\) denotes the relative discriminant ideal;
  \item \(\mathcal{N}_{M/M'}\) denotes the relative norm map;
\item \(\mathfrak{F}_{M/M'}\) denotes the conductor of \(M/M'\);
\item \(\mathcal{F}_{M/M'}\) denotes the finite part of \(\mathfrak{F}_{M/M'}\);
  \item \(\operatorname{Hil}(M)\) denotes the Hilbert class field of \(M\).
\end{itemize}

We use \(C_n\) for the cyclic group of order~\(n\).
For a Galois extension \(L/M\) and an element
\(g \in \operatorname{Gal}(L/M)\), we write
\(L^{\langle g\rangle}\) for the fixed subfield of \(L\)
under the subgroup generated by~\(g\).
By \(Z(G)\) and \(G'\), we mean the center of \(G\) and the commutator
subgroup of \(G\), respectively.

\begin{acknowledgements}
The author would like to thank Franz Lemmermeyer and Ken Ono for valuable
suggestions and feedback. Thanks are
also due to Ali Rajaei for suggesting the problem. This work was supported by
the Beijing Natural Science Foundation (Grant No.~LS24064).
\end{acknowledgements}
\section{Non-CM Galois Group Structure}
\label{sec:preliminaries}

We begin by recalling the definition of a CM field.

\begin{definition}
A number field \(K\) is called a \emph{CM-field} if it is a quadratic
extension of a totally real subfield \(K^+\) such that \(K\) is totally
imaginary, i.e.,
\([K:K^+]=2\), \(K\) is totally imaginary, and \(K^+\) is totally real.
For a totally imaginary Galois number field \(K\), being CM is equivalent to
complex conjugation being central in \(\operatorname{Gal}(K/\mathbb{Q})\).
See \cite[Lemma 2(ii)]{LOO}. A number field that is not a CM-field is called
\emph{non-CM}.
\end{definition}

\begin{lemma}\label{lem:CM-abelian}
All imaginary abelian number fields, as well as imaginary degree-\(16\) Galois
fields with Galois group isomorphic to \(C_2\times Q_8\), \(C_4 \rtimes C_4\), or
\(Q_{16}\), are CM fields.
\end{lemma}

\begin{proof}
Let \(K\) be such a field, and let
\(\tau\in \operatorname{Gal}(K/\mathbb{Q})\) be complex conjugation. Since
\(\tau\) has order \(2\), it suffices to show that every element of order \(2\)
in the listed groups is central. For abelian groups this is immediate. For the
non-abelian groups, we have the following.

\begin{enumerate}[label=\normalfont(\roman*)]
  \item For
        \(C_4 \rtimes C_4 = \langle a, b \mid a^4=b^4=1,\; bab^{-1}=a^{-1}\rangle\),
        the center is \(\langle a^2, b^2\rangle\); the only elements of order
        \(2\) are \(a^2, b^2, a^2b^2\), all of which are central.
  \item For
        \(C_2 \times Q_8 = \langle a, b, c \mid a^2 = b^4 = 1,\ c^2 = b^2,\
        ab = ba,\ ac = ca,\ cbc^{-1} = b^{-1} \rangle\),
        the center \(Z(G)=\langle a, b^2 \rangle\) contains all elements of
        order \(2\).
  \item For
        \(Q_{16} = \langle a, b \mid a^8=1,\; b^2=a^4,\; bab^{-1}=a^{-1}\rangle\),
        the center is \(\langle a^4\rangle\), which is the unique element of
        order \(2\).
\end{enumerate}

Thus \(\tau\) is central. Its fixed field \(K^{\langle \tau\rangle}\) is
therefore a Galois extension of \(\mathbb{Q}\) of degree
\([K:\mathbb{Q}]/2\). Since \(K\) is totally imaginary, this fixed field is
totally real. Hence \(K\) is a CM field.
\end{proof}

The groups of order \(16\) are
\[
C_{16},\quad C_4\times C_4,\quad C_2\times C_8,\quad C_2^2\times C_4,\quad C_2^4,\quad C_2\times Q_8,\quad
C_4 \rtimes C_4,\quad Q_{16},
\]
\[
D_{16},\quad SD_{16},\quad M_4(2),\quad
C_4\circ D_8,\quad C_2 \times D_8,\quad C_2^2 \rtimes C_4.
\]
By Lemma~\ref{lem:CM-abelian}, the groups
\[
C_{16},\quad C_4\times C_4,\quad C_2\times C_8,\quad C_2^2\times C_4,\quad C_2^4,\quad C_2\times Q_8,\quad
C_4 \rtimes C_4,\quad Q_{16}
\]
cannot occur as Galois groups of an imaginary non-CM field of degree \(16\).
Hence we obtain the following result.

\begin{theorem}\label{thm:non-cm-galois-structure}
An imaginary non-CM field of degree \(16\) has Galois group isomorphic to one
of the following six groups:
\[
D_{16},\quad SD_{16},\quad M_4(2),\quad C_4\circ D_8,\quad
C_2 \times D_8,\quad C_2^2 \rtimes C_4.
\]
\end{theorem}

\section{Ramified extensions \(L/L_{\max}\) of degree \(16\) with class number one} \label{Sec:L-ram-L_max}

We now describe the classification strategy for enumerating all imaginary non-CM fields \(L\) of degree~\(16\) with class number one for which \(L/L_{\max}\) is ramified at some finite primes, with Galois groups as listed in Theorem~\ref{thm:non-cm-galois-structure}.
The enumeration strategy is based on constructing degree-\(16\) fields \(L\) as extensions of their maximal normal octic subfields \(L_{\max}\), where both \(L\) and \(L_{\max}\) are imaginary and both are abelian over a suitable quadratic field \(N\). Moreover, if \(L\) is imaginary and non-CM, then \(L_{\max}\) is imaginary: if
\(L_{\max}\) were totally real, then \(L\) would be a CM field, contrary to the
hypothesis.

\begin{enumerate}[label=\normalfont(\alph*)]
    \item \textbf{Identify Octic Base Fields:} Begin with the known lists of imaginary degree-\(8\) fields \(L_{\max}\) that can serve as the maximal normal subfields for the target degree-\(16\) Galois groups.

    \item \textbf{Select an Extension Base:} The target field \(L\) is constructed as an abelian extension of a quadratic subfield of \(L_{\max}\), denoted \(N\). Thus \(L\subseteq \mathcal{R}_{\mathfrak{F}_{L/N}}(N)\). To build \(\mathcal{F}_{L/N}\), we use the known conductor of the octic subfield \(L_{\max}\) over \(N\).

    \item \textbf{Determine the Conductor's Odd Part:} The odd part of the conductor of the extension \(L/N\), denoted \(\mathcal{F}_{L/N}^{\mathrm{odd}}\), is derived from the conductor of the known extension \(L_{\max}/N\). We aim to show that, in many situations---the possible exception being \(C_4\circ D_8\)---the following equality holds:
    \[
    \mathcal{F}_{L/N}^{\mathrm{odd}} = \mathcal{F}_{L_{\max}/N}^{\mathrm{odd}}.
    \]

    \item \textbf{Bound the Conductor's Even Part:} Using the methods of \(2\)-adic ramification (Theorem~\ref{thm:gen_2adic} with \(m=4\)), establish a theoretical upper bound \(E_{\max}\) for the exponent \(e_{\mathfrak{p}}\) of each prime ideal \(\mathfrak{p} \mid 2\) appearing in the conductor's even part, \(\mathcal{F}_{L/N}^{\mathrm{even}} = \prod_{\mathfrak{p}\mid 2} \mathfrak{p}^{e_{\mathfrak{p}}}\).

    \item \textbf{Construct Candidate Fields:} Build the set of candidate conductors by combining the odd part with all admissible even parts up to the established bound:
    \[
    \mathcal{F}_{L/N} = \mathcal{F}_{L/N}^{\mathrm{odd}} \cdot \prod_{\mathfrak{p}\mid 2} \mathfrak{p}^{e_{\mathfrak{p}}},
    \qquad \text{where } 0 \le e_{\mathfrak{p}} \le E_{\max}.
    \]
    The desired degree-\(16\) fields \(L\) are obtained as subfields of the ray class fields \(\mathcal{R}_{\mathfrak{F}_{L/N}}(N)\).

    \item \textbf{Verify and Filter:} From the resulting set of candidate fields, retain only those satisfying:
    \begin{itemize}
        \item \(\operatorname{Gal}(L/\mathbb{Q})\) is isomorphic to the target group;
        \item \(h_L = 1\).
    \end{itemize}
\end{enumerate}

\subsection{The Even Part of the Conductor (Behavior of Prime(s) above \(2\) in the Conductor)}

\begin{theorem}[2-adic conductor bound]\label{thm:gen_2adic}
Let \(N\) be a number field, and let \(L/N\) be an abelian extension of degree \(2^m\) with \(m \ge 1\).
Suppose there exists an intermediate field \(L_{\max}\) such that
\[
N \subset L_{\max} \subset L, \qquad [L:L_{\max}]=2, \qquad [L_{\max}:N]=2^{m-1}.
\]
Let \(\mathfrak{q}\) be a prime ideal of \(\mathcal{O}_N\) above \(2\). Choose a prime ideal \(\mathfrak{q}_{L_{\max}}\) of \(\mathcal{O}_{L_{\max}}\) lying above \(\mathfrak{q}\), and define
\[
e := e(\mathfrak{q}_{L_{\max}} / \mathfrak{q}), \qquad
f := f(\mathfrak{q}_{L_{\max}} / \mathfrak{q}).
\]
Then the conductor \(\mathcal{F}_{L/N}\) satisfies the following uniform bound:
\begin{equation}\label{eq:s_bound}
v_{\mathfrak{q}}(\mathcal{F}_{L/N})
\;\le \left\lfloor
\frac{f\left( 2e - 1 + 2e\, e_{\mathfrak{q}}\, v_2(2e) \right)}{2^{m-1}}
\right\rfloor,
\tag{1}
\end{equation}
where \(v_2\) denotes the \(2\)-adic valuation.
\end{theorem}

\begin{proof}
By the conductor--discriminant formula for abelian extensions (see \cite[Ch.~3, Exercise~15(a)]{CohenAdvance}), we have
\[
\mathcal{F}_{L/N}^{\,2^{m-1}} \mid \mathfrak{d}_{L/N}.
\]
Taking \(\mathfrak{q}\)-adic valuations gives
\begin{equation}\label{eq:s_bound-cond-disc}
v_{\mathfrak{q}}(\mathcal{F}_{L/N})
\le
\left\lfloor
\frac{v_{\mathfrak{q}}(\mathfrak{d}_{L/N})}{2^{m-1}}
\right\rfloor.
\tag{2}
\end{equation}

To obtain the explicit bound \eqref{eq:s_bound}, choose primes \(\mathfrak{q}_L\) and
\(\mathfrak{q}_{L_{\max}}\) lying above \(\mathfrak{q}\) in \(L\) and \(L_{\max}\), respectively. Applying the
discriminant-different formula
\(\mathfrak{d}_{L/N} = N_{L/N}(\mathfrak{D}_{L/N})\) in conjunction with
Serre's upper bound for wild ramification in \(2\)-group extensions
\cite[Ch.~III, \S6, Prop.~13]{Serre}, we have
\[
v_{\mathfrak{q}}(\mathfrak{d}_{L/N})
= f(L/N)\, v_{\mathfrak{q}_L}(\mathfrak{D}_{L/N})
\le f(L/N)\left( e(L/N) - 1 + e(L/N)\, e_{\mathfrak{q}}\, v_2(e(L/N)) \right).
\]
Since \([L:L_{\max}] = 2\), the tower \(N \subset L_{\max} \subset L\) gives
either \(e(L/N) = 2e\), \(f(L/N) = f\), or \(e(L/N) = e\), \(f(L/N) \le 2f\). Substituting both possible values and
comparing them yields
\[
v_{\mathfrak{q}}(\mathfrak{d}_{L/N})
\le f\left( 2e - 1 + 2e\, e_{\mathfrak{q}}\, v_2(2e) \right).
\]
Combining this estimate with (2) gives
\[
v_{\mathfrak{q}}(\mathcal{F}_{L/N})
\le
\left\lfloor
\frac{f\left( 2e - 1 + 2e\, e_{\mathfrak{q}}\, v_2(2e) \right)}{2^{m-1}}
\right\rfloor,
\]
which is \eqref{eq:s_bound}.
\end{proof}

\subsection{The Odd Part of the Conductor \(\mathcal{F}_{L/N}\)}

As explained above, we identify suitable conductors for \(L\) over a quadratic
subfield \(N \subset L\) such that \(L/N\) is an abelian extension of degree
\(8\). This enables the computation and classification of \(L\) as a ray class
field over \(N\). In the previous section, we explained how to find the even
part of \(\mathcal{F}_{L/N}\) from \(\mathcal{F}_{L_{\max}/N}\). Now we do the
same for the odd part of \(\mathcal{F}_{L/N}\).

First, we note that any odd prime ramifying in a degree-\(16\) field in the list
of Theorem~\ref{thm:non-cm-galois-structure} must already ramify in the normal
degree-\(8\) field \(L_{\max}\) defined below. For this purpose, we first need
to show that every quadratic subfield of \(L\) is contained in \(L_{\max}\).

\begin{proposition}\label{prop:quadratic_in_Lmax}
Let \(L/\mathbb{Q}\) be a Galois extension with Galois group
\[
G \in \left\{ D_{16},\; SD_{16},\; M_4(2),\; C_4\circ D_8,\;
C_2 \times D_8,\; C_2^2 \rtimes C_4 \right\},
\]
and choose the degree-\(8\) subfield \(L_{\max}\subset L\) according to the
following cases.

\begin{itemize}
\item If \(G \cong D_{16} = \langle a,b \mid a^8=b^2=1,\; bab^{-1}=a^{-1}\rangle\),
      take \(L_{\max}=L^{\langle a^4\rangle}\), with
      \[
      \operatorname{Gal}(L_{\max}/\mathbb{Q}) \cong D_8.
      \]

\item If \(G \cong SD_{16} = \langle a,b \mid a^8=b^2=1,\; bab^{-1}=a^3\rangle\),
      take \(L_{\max}=L^{\langle a^4\rangle}\), with
      \[
      \operatorname{Gal}(L_{\max}/\mathbb{Q}) \cong D_8.
      \]

\item If \(G \cong M_4(2) = \langle a,b \mid a^8=b^2=1,\; bab=a^5\rangle\),
      take \(L_{\max}=L^{\langle a^4\rangle}\), with
      \[
      \operatorname{Gal}(L_{\max}/\mathbb{Q}) \cong C_4 \times C_2.
      \]

\item If \(G \cong C_4 \circ D_8 =
      \langle a,b,c \mid a^4=c^2=1,\; b^2=a^2,\; ab=ba,\; ac=ca,\;
      cbc=b^{-1}\rangle\), take \(L_{\max}=L^{\langle a^2\rangle}\), with
      \[
      \operatorname{Gal}(L_{\max}/\mathbb{Q}) \cong C_2 \times C_2 \times C_2.
      \]

\item If \(G \cong C_2 \times D_8 =
      \langle a,b,c \mid a^2=b^4=c^2=1,\; ab=ba,\; ac=ca,\; cbc=b^{-1}\rangle\),
      take \(L_{\max}=L^{\langle b^2\rangle}\), with
      \[
      \operatorname{Gal}(L_{\max}/\mathbb{Q}) \cong C_2 \times C_2 \times C_2.
      \]

\item If \(G \cong C_2^2 \rtimes C_4 =
      \langle a,b,c \mid a^2=b^2=c^4=1,\; ab=ba,\; cac^{-1}=ab,\; bc=cb\rangle\),
      take \(L_{\max}=L^{\langle b\rangle}\). Then
      \[
      \operatorname{Gal}(L_{\max}/\mathbb{Q}) \cong C_2 \times C_4.
      \]
\end{itemize}

Then every quadratic subfield of \(L\) is contained in \(L_{\max}\).
\end{proposition}
\begin{proof}
By the fundamental theorem of Galois theory, quadratic subfields of \(L\)
correspond bijectively to subgroups \(H\le G\) of index \(2\), via \(N=L^H\).
Moreover, \(L_{\max}=L^W\) contains \(N=L^H\) if and only if \(W\subseteq H\),
where \(W:=\operatorname{Gal}(L/L_{\max})\). Thus it suffices to show that \(W\)
is contained in \emph{every} subgroup of index \(2\) in \(G\).

Since any subgroup of index \(2\) is normal with quotient \(G/H\cong C_2\)
abelian, the derived subgroup \(G'\) satisfies \(G'\subseteq H\) for every
index-\(2\) subgroup \(H\). Hence it is enough to verify \(W\subseteq G'\). We
treat each case.

For \(G\cong D_{16}\) and \(SD_{16}\), we have \(W=\langle a^4\rangle\) and
\(G'=\langle a^2\rangle\), so
\[
W=\langle a^4\rangle\subseteq\langle a^2\rangle=G'.
\]
For \(G\cong M_4(2)\), we have \(W=\langle a^4\rangle=G'\).

For \(G\cong C_4\circ D_8\), we have \(W=\langle a^2\rangle=G'\).
For \(G\cong C_2\times D_8\), we have \(W=\langle b^2\rangle=G'\).
For \(G\cong C_2^2\rtimes C_4\), we have \(W=\langle b\rangle=G'\).

In every case \(W\subseteq G'\subseteq H\) for each index-\(2\) subgroup \(H\)
of \(G\). By the Galois correspondence, this yields
\[
L^H\subseteq L^W=L_{\max}
\]
for every quadratic subfield \(L^H\) of \(L\), as required.
\end{proof}

\begin{remark}\label{rem:alternative-Lmax-C22C4}
In the case \(G\cong C_2^2\rtimes C_4\), the Frattini subgroup is
\[
\Phi(G)=\langle b, c^2\rangle\cong C_2\times C_2,
\]
of order \(4\). Since \(\Phi(G)\) is the intersection of all index-\(2\)
subgroups of \(G\), any subgroup \(H\le\Phi(G)\) of order \(2\) yields a
degree-\(8\) subfield \(L_{\max}=L^H\) containing all quadratic subfields of
\(L\). The three such subgroups are
\[
\langle b\rangle,\qquad \langle c^2\rangle,\qquad \langle bc^2\rangle.
\]
The first gives \(\operatorname{Gal}(L_{\max}/\mathbb{Q})\cong C_2\times C_4\)
(the choice made in Proposition~\ref{prop:quadratic_in_Lmax}), while the last
two give \(\operatorname{Gal}(L_{\max}/\mathbb{Q})\cong D_8\). Consequently,
the choice of \(L_{\max}\) is not unique in this case, but the conclusion that
all quadratic subfields lie in \(L_{\max}\) holds for all three choices.
\end{remark}

Now we are ready to descend the ramification of odd primes in \(L\) to
\(L_{\max}\).

\begin{proposition}\label{prop:no_new_odd}
Let \(L\) be an imaginary number field of odd class number, and let \(L_{\max}\)
be a maximal subfield of index \([L:L_{\max}]=2\) such that all quadratic
subfields of \(L\) are contained in \(L_{\max}\). Then, apart from the primes
dividing \(2\), every odd prime \(p\) that ramifies in the quadratic extension
\(L/L_{\max}\) must already ramify in \(L_{\max}/\mathbb{Q}\).
\end{proposition}

\begin{proof}
We argue by contradiction. Suppose there exists an odd prime \(p\) that is
ramified in the quadratic extension \(L/L_{\max}\), yet unramified in
\(L_{\max}/\mathbb{Q}\).

Define the fundamental quadratic discriminant parameter
\[
p^*=(-1)^{\frac{p-1}{2}}\,p=
\begin{cases}
p, & p\equiv 1 \pmod 4,\\[1ex]
-p, & p\equiv 3 \pmod 4.
\end{cases}
\]
The quadratic field \(k=\mathbb{Q}(\sqrt{p^*})\) has discriminant \(p^*\) and
is ramified solely at \(p\).

First, observe that \(\sqrt{p^*}\notin L\). Indeed, if \(\sqrt{p^*}\in L\),
then \(k\subseteq L\). Since all quadratic subfields of \(L\) are contained in
\(L_{\max}\), this would force \(\sqrt{p^*}\in L_{\max}\) (by
Proposition~\ref{prop:quadratic_in_Lmax}), contradicting the hypothesis that
\(p\) is unramified in \(L_{\max}/\mathbb{Q}\).

Now consider the compositum \(K=L(\sqrt{p^*})\). The field \(K\) is a bicyclic
biquadratic extension of \(L_{\max}\).

The extension \(K/L=L(\sqrt{p^*})/L\) is unramified at all finite primes
\(\mathfrak{q}\nmid p\). Moreover, since \(L\) is totally imaginary, all
infinite places in \(K/L\) are complex and hence trivially unramified.

It remains to check the primes above \(p\). Let \(\mathfrak{p}\) be a prime
ideal of \(L_{\max}\) above \(p\), and let \(\mathfrak{P}\) be a prime ideal of
\(K\) above \(\mathfrak{p}\). Since \(p\) is odd, the ramification at
\(\mathfrak{p}\) is tame, so the inertia group
\(I(\mathfrak{P}/\mathfrak{p})\) is cyclic. As
\(I(\mathfrak{P}/\mathfrak{p})\subseteq
\operatorname{Gal}(K/L_{\max})\cong C_2\times C_2\), it follows that
\(|I(\mathfrak{P}/\mathfrak{p})|\le 2\). Since, by assumption,
\(\mathfrak{p}\) is ramified in \(L/L_{\max}\), it follows that
\(K/L\) is unramified at every prime ideal above \(p\).

Thus \(K/L\) is a nontrivial, everywhere unramified abelian extension. By class
field theory, the Hilbert class field of \(L\) contains \(K\), which implies
that \(2\mid h(L)\). This contradicts the hypothesis that \(h(L)\) is odd.

Therefore, every odd
prime ramifying in \(L/L_{\max}\) must already ramify in
\(L_{\max}/\mathbb{Q}\).
\end{proof}

From Proposition~\ref{prop:quadratic_in_Lmax},
Remark~\ref{rem:alternative-Lmax-C22C4}, and
Proposition~\ref{prop:no_new_odd}, we obtain:

\begin{theorem}\label{thm:no-new-odd}
For imaginary non-CM fields of class number one with Galois group in
\[
\{D_{16}, SD_{16}, M_4(2), C_2^2\rtimes C_4, C_2\times D_8, C_4 \circ D_8\},
\]
every odd prime that ramifies in \(L\) already ramifies in \(L_{\max}\).
\end{theorem}

\subsection{Imaginary non-CM \(D_{16}\) and \(SD_{16}\) fields of class number
one, ramified over the unique \(D_8\) subfield}

We first recall the following lemma, which will be used frequently.

\begin{lemma}\label{lem:masley} \cite[Corollary 2.2]{JMasley}
For a finite extension \(M/M'\), one has
\[
h_{M'}\mid [M:M']\,h_M.
\]
Moreover, if \(M/M'\) is totally ramified at a finite place or ramified at an
infinite place, then
\[
h_{M'}\mid h_M.
\]
\end{lemma}

From Lemma~\ref{lem:masley}, we obtain the following corollary.

\begin{corollary}\label{Cor:h-Lmax}
If \(L/L_{\max}\) is ramified at some finite place and \(h_L=1\), then
\(h_{L_{\max}}=1\).
\end{corollary}

\begin{proof}
Since \(L/L_{\max}\) has degree \(2\), ramification at a finite place implies
total ramification. Hence, by Lemma~\ref{lem:masley},
\(h_{L_{\max}}\mid h_L\). Since \(h_L=1\), we get \(h_{L_{\max}}=1\).
\end{proof}

\begin{proposition}\label{prop:odd_prime_cyclic}
Let \(L_k\subseteq L_{\max}\subseteq L\) be intermediate fields, with
\(L/\mathbb{Q}\) Galois of \(2\)-power degree. Suppose that every cyclic
subgroup of \(\operatorname{Gal}(L/\mathbb{Q})\) of order at least \(4\) is
contained in \(\operatorname{Gal}(L/L_k)\). If an odd prime \(p\) ramifies in
\(L_k/\mathbb{Q}\), then no prime ideal of \(L_{\max}\) above \(p\) ramifies in
\(L/L_{\max}\).
\end{proposition}

\begin{proof}
Suppose, for the sake of contradiction, that an odd prime \(p\) ramifies in
\(L_k/\mathbb{Q}\) and that some prime ideal of \(L_{\max}\) above \(p\)
ramifies in \(L/L_{\max}\). Fix a prime ideal \(\mathfrak{P}\) of \(L\) lying
above \(p\), and let
\[
I=I(\mathfrak{P}/p)\le \operatorname{Gal}(L/\mathbb{Q})
\]
denote its inertia group.

Since \([L:\mathbb{Q}]\) is a power of \(2\) and \(p>2\), the extension
\(L/\mathbb{Q}\) is tamely ramified at \(p\). Consequently, \(I\) is cyclic,
and \(|I|=e(\mathfrak{P}/p)\).

Set \(\mathfrak{p}_{L_{\max}}=\mathfrak{P}\cap L_{\max}\) and
\(\mathfrak{p}_{L_k}=\mathfrak{P}\cap L_k\). Because
\(\mathfrak{p}_{L_{\max}}\) ramifies in \(L/L_{\max}\), we have
\[
e(\mathfrak{P}/\mathfrak{p}_{L_{\max}})=2.
\]
Because \(p\) ramifies in \(L_k/\mathbb{Q}\), we have
\(e(\mathfrak{p}_{L_k}/p)\ge 2\). By multiplicativity of ramification indices
in the tower \(\mathbb{Q}\subset L_k\subseteq L_{\max}\subset L\), we obtain
\[
|I|=e(\mathfrak{P}/p)
=e(\mathfrak{P}/\mathfrak{p}_{L_{\max}})\cdot
 e(\mathfrak{p}_{L_{\max}}/\mathfrak{p}_{L_k})\cdot
 e(\mathfrak{p}_{L_k}/p)
\ge 2\cdot 1\cdot 2=4.
\]
Let \(J=L^I\) be the fixed field of \(I\). Then \(L/J\) is cyclic of degree
\(|I|\ge 4\).

By hypothesis, the cyclic subgroup \(I=\operatorname{Gal}(L/J)\) of order at
least \(4\) is contained in \(\operatorname{Gal}(L/L_k)\). Hence, by the
Galois correspondence,
\[
L_k\subseteq J.
\]
But \(J\) is the fixed field of the inertia group at \(p\),
so \(p\) is unramified in \(J/\mathbb{Q}\). This contradicts the assumption
that \(p\) is ramified in \(L_k/\mathbb{Q}\).

Therefore, no prime ideal of \(L_{\max}\) above \(p\) ramifies in
\(L/L_{\max}\).
\end{proof}

Moreover, \cite[Ch.~3, Prop.~3.3.18]{CohenAdvance} yields the following corollary.

\begin{corollary}\label{cor:odd_squarefree}
Let \(N\) be a quadratic subfield of \(L\) such that \(L/N\) is abelian. For
the cyclic \(2\)-power extensions \(L/N\) and \(L_{\max}/N\), every odd prime
ideal dividing their conductors occurs with exponent exactly \(1\).
\end{corollary}

\begin{proposition}\label{prop:cyclic_quadratic}
Let \(L/\mathbb{Q}\) be a Galois extension with Galois group \(G\).
\begin{enumerate}[label=\normalfont(\alph*)]
    \item If \(G\cong D_{16}\) or \(SD_{16}\), there is a unique cyclic
    subgroup of order \(8\), say \(\langle a\rangle\le G\). The fixed field
    \[
    N=L^{\langle a\rangle}
    \]
    is a quadratic subfield of \(L\), and \(L/N\) is cyclic of degree \(8\).
    \item If \(G\cong C_2\times D_8=\langle a,b,c\mid a^2=b^4=c^2=1,\;
    ab=ba,\; ac=ca,\; cbc=b^{-1}\rangle\), the only cyclic subgroups of order
    \(4\) in \(G\) are \(\langle b\rangle\) and \(\langle ab\rangle\). They
    generate the abelian subgroup
    \[
    M=\langle b,ab\rangle=\langle a,b\rangle\cong C_2\times C_4
    \]
    of index \(2\). The fixed field \(N=L^M=L^{\langle a,b\rangle}\) is a
    quadratic subfield of \(L\), and \(L/N\) is abelian of degree \(8\).
\end{enumerate}
\end{proposition}

\begin{theorem}\label{thm:odd_ramification_union}
Let \(L/\mathbb{Q}\) be a Galois extension with Galois group
\(G\in\{D_{16}, SD_{16}, C_2\times D_8\}\), and let \(N\) be the quadratic
subfield defined in Proposition~\ref{prop:cyclic_quadratic}. Then
\[
\mathcal{F}_{L/N}^{\mathrm{odd}}=\mathcal{F}_{L_{\max}/N}^{\mathrm{odd}}.
\]
\end{theorem}

\begin{proof}
Let \(p\) be an odd prime. By Corollary~\ref{cor:odd_squarefree}, each odd
prime ideal dividing \(\mathcal{F}_{L/N}\) or \(\mathcal{F}_{L_{\max}/N}\)
occurs with exponent \(1\). Hence it suffices to show that the supports of the
two odd conductors coincide.

Suppose first that \(p\) ramifies in \(L/N\). If some prime ideal of
\(L_{\max}\) above \(p\) ramifies in \(L/L_{\max}\), then
Proposition~\ref{prop:odd_prime_cyclic} (applied with \(L_k=N\)) implies that
\(p\) cannot ramify in \(N/\mathbb{Q}\). Since \(N/\mathbb{Q}\) is quadratic,
this means that \(p\) is unramified in \(N/\mathbb{Q}\) but ramified in
\(L_{\max}/\mathbb{Q}\). Therefore \(p\) ramifies in \(L_{\max}/N\).

If no prime ideal of \(L_{\max}\) above \(p\) ramifies in \(L/L_{\max}\), then
the ramification of \(p\) in \(L/N\) must already occur in \(L_{\max}/N\).
Thus \(p\) ramifies in \(L_{\max}/N\) in either case.

Conversely, if \(p\) ramifies in \(L_{\max}/N\), then \(p\) certainly ramifies
in \(L/N\). Hence the odd prime supports coincide, and therefore
\[
\mathcal{F}_{L/N}^{\mathrm{odd}}=\mathcal{F}_{L_{\max}/N}^{\mathrm{odd}},
\]
as desired.
\end{proof}

\subsection{Imaginary non-CM \(C_2^2\rtimes C_4\) and \(M_4(2)\) fields of
class number one, ramified over their maximal \(C_2\times C_4\) subfield}

\begin{proposition}\label{prop:C2C2n-ram-FN}
Let \(n\ge 2\), and let \(F/\mathbb{Q}\) be an abelian extension with
\[
G=\operatorname{Gal}(F/\mathbb{Q})\cong C_2\times C_{2^n}.
\]
Let \(H\le G\) be a subgroup containing the \(2\)-torsion subgroup
\[
G[2]=\{g\in G\mid g^2=1\}\cong C_2\times C_2,
\]
and let \(N=F^H\) be its fixed field.

Then every odd prime \(p\) that ramifies in \(F/\mathbb{Q}\) also ramifies in
the relative extension \(F/N\). Consequently, the odd prime supports of the
conductors \(\mathcal{F}_{F/\mathbb{Q}}\) and \(\mathcal{F}_{F/N}\) coincide.
\end{proposition}

\begin{proof}
Let \(p\) be an odd prime that ramifies in \(F/\mathbb{Q}\). Let \(I_p\le G\)
denote the inertia subgroup of a prime of \(F\) lying above \(p\).

Since \(p\) is odd and \([F:\mathbb{Q}]=2^{n+1}\) is a power of \(2\), the
ramification is tame, so \(I_p\) is a nontrivial cyclic \(2\)-subgroup:
\[
I_p\cong C_{2^k}\quad \text{for some } 1\le k\le n.
\]
By Galois theory, \(p\) ramifies in \(F/N\) if and only if
\[
I_p(F/N)=I_p\cap H\neq\{1\}.
\]
Since \(I_p\) is a nontrivial cyclic \(2\)-group, it contains a unique subgroup
of order \(2\), generated by its unique involution \(\tau\). Thus
\(\tau\in G[2]\). By hypothesis \(G[2]\subseteq H\), hence \(\tau\in H\), and
therefore
\[
\tau\in I_p\cap H,
\]
so \(I_p\cap H\neq\{1\}\). It follows that every odd prime \(p\) ramifying in
\(F/\mathbb{Q}\) also ramifies in \(F/N\).
Applying Corollary~\ref{cor:odd_squarefree}, we conclude that the odd parts of
the conductors are equal.
\end{proof}

\begin{proposition}\label{prop:qua-subfield}
Let \(L/\mathbb{Q}\) be a Galois extension whose Galois group
\(G=\operatorname{Gal}(L/\mathbb{Q})\) is isomorphic to either
\[
M_4(2)\quad \text{or}\quad C_2^2\rtimes C_4.
\]
Let \(L_{\max}=L^{G'}\) denote the maximal abelian subextension of
\(L/\mathbb{Q}\) of degree \(8\), and let \(N\subset L_{\max}\) be the
quadratic subfield defined according to the isomorphism type as follows.
\begin{enumerate}
    \item[\textup{(1)}] If
    \(G\cong M_4(2)=\langle a,b\mid a^8=b^2=1,\; bab=a^5\rangle\), then
    \[
    G'=\langle a^4\rangle, \qquad N=L^{\langle a^2,b\rangle}.
    \]
    \item[\textup{(2)}] If
    \(G\cong C_2^2\rtimes C_4=\langle a,b,c\mid a^2=b^2=c^4=1,\; ab=ba,\;
    cac^{-1}=ab,\; bc=cb\rangle\), then
    \[
    G'=\langle b\rangle, \qquad N=L^{\langle a,b,c^2\rangle}.
    \]
\end{enumerate}
Then:
\begin{enumerate}
    \item[\textup{(i)}]  \(L/N\) is abelian.
    \item[\textup{(ii)}] 
    \(\operatorname{Gal}(L_{\max}/\mathbb{Q})\cong C_2\times C_4\), and every
    rational prime \(p\) that ramifies in \(L_{\max}/\mathbb{Q}\) ramifies in
    \(L_{\max}/N\).
    \item[\textup{(iii)}] In both fields,
    \[
    \mathcal{F}_{L/N}^{\mathrm{odd}}=\mathcal{F}_{L_{\max}/N}^{\mathrm{odd}}.
    \]
\end{enumerate}
\end{proposition}

\begin{proof}
In both situations \(L_{\max}=L^{G'}\). A subfield \(K\subset L_{\max}\) with
\([K:\mathbb{Q}]=4\) corresponds to a normal subgroup \(U\trianglelefteq G\) of
order \(4\) containing \(G'\) such that \(G/U\cong C_2^2\). More precisely,
\(U=\langle a^2\rangle\) when \(G\cong M_4(2)\), and
\(U=\langle b,c^2\rangle\) when \(G\cong C_2^2\rtimes C_4\). The quadratic
subfield \(N\subset K\) corresponds to a subgroup \(M\le G\) of index \(2\)
containing \(U\), and \(\operatorname{Gal}(L/N)\cong M\).

\emph{(i)} In both situations, \(M\) is abelian. Indeed,
\(M=\langle a^2,b\rangle\) for \(G\cong M_4(2)\), while
\(M=\langle a,b,c^2\rangle\) for \(G\cong C_2^2\rtimes C_4\). Hence \(L/N\) is
abelian.

\emph{(ii)} By Proposition~\ref{prop:quadratic_in_Lmax}, for both isomorphism
types \(\operatorname{Gal}(L_{\max}/\mathbb{Q})\cong C_2\times C_4\), and
\(\operatorname{Gal}(L_{\max}/N)\) is precisely the \(2\)-torsion subgroup of
\(\operatorname{Gal}(L_{\max}/\mathbb{Q})\):
\[
\operatorname{Gal}(L_{\max}/N)
=(\operatorname{Gal}(L_{\max}/\mathbb{Q}))[2]\cong C_2\times C_2.
\]
Applying Proposition~\ref{prop:C2C2n-ram-FN} with \(F=L_{\max}\) and \(n=2\),
every odd prime ramifying in \(L_{\max}/\mathbb{Q}\) also ramifies in
\(L_{\max}/N\).

\emph{(iii)} By Theorem~\ref{thm:no-new-odd}, every odd prime ramifying in
\(L\) already ramifies in \(L_{\max}\). Combined with part (ii), the odd prime
supports of \(\mathcal{F}_{L/N}\) and \(\mathcal{F}_{L_{\max}/N}\) coincide.
Since both conductors are squarefree at odd primes by
Corollary~\ref{cor:odd_squarefree}, we obtain
\[
\mathcal{F}_{L/N}^{\mathrm{odd}}=\mathcal{F}_{L_{\max}/N}^{\mathrm{odd}}.
\]
\end{proof}

\begin{example}[Finding \(M_4(2)\) fields]\label{exmp:code-M4(2)}
We illustrate this procedure by searching for imaginary non-CM fields of
degree \(16\) with Galois group \(M_4(2)\) and class number one.

The group \(M_4(2)\) has a quotient isomorphic to \(C_2\times C_4\). Therefore
the search begins with a known imaginary field \(L_{\max}\) of degree \(8\)
having \(\operatorname{Gal}(L_{\max}/\mathbb{Q})\cong C_2\times C_4\) and
\(h_{L_{\max}}=1\). For instance, we take
\(L_{\max}=\mathbb{Q}(\zeta_{16})\) (see \cite{Yamamura1}). Using
PARI/GP~\cite{Pari/gp}, the imaginary quadratic subfields of \(L_{\max}\) are
found to be \(\mathbb{Q}(\sqrt{-1})\) and \(\mathbb{Q}(\sqrt{-2})\). We choose
\(N=\mathbb{Q}(\sqrt{-2})\).

The relative conductor is
\[
\mathfrak{F}_{L_{\max}/N}=\mathfrak{q}^{5},
\]
where \(\mathfrak{q}=(\sqrt{-2})\mathcal{O}_N\) is the unique prime above
\(2\). Since \(\mathcal{F}_{L_{\max}/N}^{\mathrm{odd}}=(1)\), the conductor of
the target extension is purely \(2\)-adic, of the form
\(\mathfrak{F}_{L/N}=\mathfrak{q}^s\) with \(s\le E_{\max}\). Using PARI/GP to
compute the ray class fields \(\mathcal{R}_{\mathfrak{q}^s}(N)\), we extract
the degree-\(8\) abelian subextensions over \(N\), compute their Galois
closures over \(\mathbb{Q}\), and verify their class numbers, retaining only
those with \(\operatorname{Gal}(L/\mathbb{Q})\cong M_4(2)\) and \(h_L=1\).
This process is repeated systematically for all admissible \(L_{\max}\) and
their quadratic subfields \(N\).
\end{example}
\begin{remark}\label{remark:Only-non-CM}
The algorithm in Example~\ref{exmp:code-M4(2)}, implemented in the
accompanying code~\cite{AmiriCode2026}, produces only imaginary non-CM
fields of class number one; it does not produce CM fields. This is because
only two groups, \(D_{16}\) and \(C_4\circ D_8\), arise as Galois groups of
both CM fields and imaginary non-CM fields of degree \(16\) with class number
one; see Appendix~\ref{AppendixA} and
Theorem~\ref{thm:non-cm-galois-structure}. For \(D_{16}\), however, the two
classes do not overlap. A \(D_{16}\)-field has a unique maximal normal
subfield \(L_{\max}\) of degree \(8\). If \(L_{\max}\) is real, then its
quadratic subfield \(N\subset L_{\max}\) is real as well, and
\(\operatorname{Gal}(L/N)\) is cyclic. For an imaginary non-CM field, by
contrast, both \(N\) and \(L_{\max}\) are imaginary. (Here \(N\) is imaginary
because \(L/N\) is cyclic and complex conjugation lies in
\(\operatorname{Gal}(L/L_{\max})\).) If
\(\operatorname{Gal}(L/\mathbb{Q})\cong C_4\circ D_8\), there is only one CM
field of class number one. When comparing our results with the known list,
one must therefore test for isomorphism with this field and exclude it from
the imaginary non-CM fields of class number one.
\end{remark}

\subsection{Imaginary non-CM \(C_4\circ D_8\) fields of class number one,
ramified over the unique \(C_2^3\) subfield}

Let \(L/\mathbb{Q}\) be a Galois extension with Galois group
\[
G=\operatorname{Gal}(L/\mathbb{Q})\cong C_4\circ D_8,
\]
which has order \(16\) and a presentation
\[
G=\langle a,r,s\mid a^4=r^4=s^2=1,\; srs=r^{-1},\;
[a,r]=[a,s]=1,\; a^2=r^2\rangle.
\]
Define
\[
L_{\max}=L^{\langle r^2\rangle},
\]
and let \(N=L^{\langle a,r\rangle}\) be the quadratic subfield such that
\(L/N\) is abelian. (Note that \(\langle r^2\rangle\subseteq\langle a,r\rangle\),
so \(N\subseteq L_{\max}\subseteq L\).)

Since both \(L\) and \(L_{\max}\) are abelian over \(N\), and
\(L_{\max}\subset L\), the conductor-discriminant formula, together with class
field theory, ensures that
\[
\mathcal{F}_{L_{\max}/N}\mid\mathcal{F}_{L/N}.
\]
Moreover, by Theorem~\ref{thm:no-new-odd}, every prime ideal of \(N\) that
ramifies in \(L/N\) lies above a rational prime that already ramifies in
\(L_{\max}/\mathbb{Q}\). Therefore
\[
\mathcal{F}_{L/N}=\mathcal{F}_{L_{\max}/N}
\cdot\prod_{m}\mathfrak{p}_m^{\,j_m},
\]
where the \(\mathfrak{p}_m\) are the prime ideals of \(N\) above the odd primes
dividing \(\delta_{N/\mathbb{Q}}\). In other words, \(\prod_m\mathfrak{p}_m\) is
the product of the prime ideals of \(N\) above the odd divisors of the
discriminant \(\delta_{N/\mathbb{Q}}\).

\section{Unramified extensions \(L/L_{\max}\) of degree \(16\) with class number one}
\label{sec:L-L_max-unram}

In this section, we determine the complete list of imaginary non-CM
degree-\(16\) number fields of class number one for which \(L/L_{\max}\) is
unramified.

We will use the following lemma frequently.

\begin{lemma}\label{lem:unram-cyclic}
Let \(p\) be a prime (possibly \(2\)), and let
\[
K\subseteq F_{\max}\subseteq F
\]
be imaginary number fields such that \(F/K\) is cyclic of degree \(p^n\) with
Galois group generated by an element \(a\), and
\[
\operatorname{Gal}(F/F_{\max})=\langle a^{p^{n-1}}\rangle.
\]
If \(F/F_{\max}\) is unramified and \(h_F=1\), then \(F/K\) is unramified at
all finite places, and \(F=\operatorname{Hil}(K)\).
\end{lemma}

\begin{proof}
Assume, for contradiction, that \(F/K\) is ramified at some finite prime.
Let \(\mathfrak{p}\) be a finite prime of \(K\) at which \(F/K\) ramifies,
and let \(\mathfrak{P}\) be a prime of \(F\) above \(\mathfrak{p}\). Since
\[
\operatorname{Gal}(F/K)=\langle a\rangle
\]
and the inertia group \(I(\mathfrak{P}/\mathfrak{p})\) is nontrivial, it
contains the unique subgroup of order \(p\) of \(\operatorname{Gal}(F/K)\),
namely \(\langle a^{p^{n-1}}\rangle\). But
\[
\langle a^{p^{n-1}}\rangle=\operatorname{Gal}(F/F_{\max}),
\]
so \(F/F_{\max}\) is ramified at the prime
\(\mathfrak{P}\cap\mathcal{O}_{F_{\max}}\) of \(F_{\max}\), contradicting the
assumption that \(F/F_{\max}\) is unramified. Thus \(F/K\) is unramified at all finite places.

Since both \(F\) and \(K\) are imaginary,
\(F\subseteq \operatorname{Hil}(K)\). Now, since \(h_F=1\), we have
\[
h_K=\#\operatorname{Cl}(K)\mid [F:K]
\]
(by Lemma~\ref{lem:masley}). Therefore, \(F=\operatorname{Hil}(K)\).
\end{proof}
\subsection{Imaginary non-CM \(M_4(2)\) and \(SD_{16}\) fields of class number one, unramified over their unique \(D_8\) subfield}

First, we set aside fields with
\[
\operatorname{Gal}(L/\mathbb{Q})\cong SD_{16}
\quad\text{or}\quad
\operatorname{Gal}(L/\mathbb{Q})\cong M_4(2).
\]
For these groups, \(L/L_{\max}\) cannot be unramified.

\begin{lemma}[Chebotarev's Monodromy Theorem]\label{lem:Monodromy}
Let \(k\subseteq K\) be number fields such that \(K/\mathbb{Q}\) and \(k/\mathbb{Q}\) are normal, and \(K/k\) is unramified at all finite places. Let \(m\) be the least positive integer such that \(g^m=1\) for every
\[
g\in\operatorname{Gal}(k/\mathbb{Q}).
\]
Then \(\operatorname{Gal}(K/\mathbb{Q})\) is generated by the elements \(\sigma\) which are not contained in \(\operatorname{Gal}(K/k)\) and satisfy \(\sigma^m=1\).
\end{lemma}

\begin{proof}
See \cite[Corollary~1]{lemmermeyer1997unramified}.
\end{proof}

\begin{theorem}[Monodromy obstruction for \(SD_{16}\) and \(M_4(2)\)]\label{thm:monodromy_obstruction}
Let \(L/\mathbb{Q}\) be an imaginary non-CM field of degree \(16\) with class number one and
\[
G\cong SD_{16}
\quad\text{or}\quad
G\cong M_4(2).
\]
Then the quadratic extension \(L/L_{\max}\) is ramified at some finite place.
\end{theorem}

\begin{proof}
When \(G\cong SD_{16}\), write
\[
SD_{16}=\langle a,b\mid a^8=b^2=1,\; bab^{-1}=a^3\rangle.
\]
The unique central subgroup of order \(2\) is \(\langle a^4\rangle\), so
\[
L_{\max}=L^{\langle a^4\rangle}.
\]
Assume, for contradiction, that \(L/L_{\max}\) is unramified. Let
\[
k=L^{\langle a\rangle}.
\]
Since \(L/k\) is cyclic of degree \(8\), any nontrivial inertia subgroup would contain \(a^4\), forcing ramification in \(L/L_{\max}\); hence \(L/k\) is unramified.

By Chebotarev's Monodromy Theorem, \(G\) must be generated by the involutions outside
\[
\operatorname{Gal}(L/k)=\langle a\rangle.
\]
The involutions outside \(\langle a\rangle\) are
\[
b,\quad a^2b,\quad a^4b,\quad a^6b.
\]
A direct calculation shows that these generate
\[
\langle a^2,b\rangle\cong D_8,
\]
which is a proper subgroup of \(SD_{16}\). Contradiction.

If \(G\cong M_4(2)\), write
\[
M_4(2)=\langle a,b\mid a^8=b^2=1,\; bab=a^5\rangle.
\]
Again
\[
L_{\max}=L^{\langle a^4\rangle},
\qquad
k=L^{\langle a\rangle}.
\]
If \(L/L_{\max}\) were unramified, then \(L/k\) would be unramified. The involutions outside \(\langle a\rangle\) are
\[
b
\quad\text{and}\quad
a^4b,
\]
which generate
\[
\langle a^4,b\rangle\cong C_2\times C_2,
\]
a proper subgroup of \(M_4(2)\). Chebotarev's theorem gives a contradiction.

Thus, in both situations, \(L/L_{\max}\) is necessarily ramified.
\end{proof}

\subsection{Imaginary non-CM \(D_{16}\)-fields of class number one, unramified over their \(D_8\) subfield}

Let the Galois group be presented as
\[
D_{16}=\langle a,b\mid a^8=b^2=1,\; bab^{-1}=a^{-1}\rangle.
\]
The subgroup \(\langle a\rangle\) is cyclic of order \(8\), and its unique element of order \(2\) is the central involution \(a^4\). We define the intermediate fixed fields
\[
L_{\max}=L^{\langle a^4\rangle},
\qquad
N=L^{\langle a\rangle}.
\]
Then
\[
[L_{\max}:\mathbb{Q}]=8,
\qquad
[N:\mathbb{Q}]=2,
\qquad
\operatorname{Gal}(L_{\max}/\mathbb{Q})\cong D_8,
\]
and \(L/N\) is a cyclic extension of degree \(8\).

\begin{theorem}\label{thm:D16-unram}
Let \(L/\mathbb{Q}\) be an imaginary non-CM \(D_{16}\)-field with class number one. If \(L/L_{\max}\) is unramified, then \(L/N\) is unramified everywhere. Consequently,
\[
L=\operatorname{Hil}(N)
\]
is the Hilbert class field of the imaginary quadratic field \(N\), and
\[
\operatorname{Cl}(N)\cong C_8.
\]
\end{theorem}

\begin{proof}
First, we verify that \(N\) is imaginary quadratic. If \(N\) were real quadratic, then since
\[
\operatorname{Gal}(L/N)=\langle a\rangle\cong C_8
\]
is cyclic of even order and \(L\) is Galois over \(\mathbb{Q}\) with
\[
\operatorname{Gal}(L/\mathbb{Q})\cong D_{16},
\]
complex conjugation would lie in \(\operatorname{Gal}(L/N)\). Hence \(L\) would be a CM field, contradicting the hypothesis that \(L\) is imaginary non-CM. Thus \(N\) is imaginary quadratic.

Since \(N\) is imaginary quadratic, it has no real places; in particular,
there is no ramification at infinity. Thus \(L/N\) is unramified everywhere.
By Lemma~\ref{lem:unram-cyclic}, it follows that
\[
L=\operatorname{Hil}(N)
\]
and
\[
\operatorname{Cl}(N)\cong C_8.
\]
\end{proof}

\bigskip

\subsection{Unramified imaginary non-CM \(C_2\times D_8\)-fields of class number \(1\)  over their maximal abelian subfield}
\label{subsec:C2D8-unram}

Let \(L/\mathbb{Q}\) be an imaginary non-CM field of degree \(16\) and class
number \(1\), with Galois group
\[
G=\operatorname{Gal}(L/\mathbb{Q})
\]
represented as
\[
C_2\times D_8
=
\langle a,b,c\mid a^2=b^4=c^2=1,\; ab=ba,\; ac=ca,\; cbc=b^{-1}\rangle.
\]
The derived subgroup is
\[
G'=\langle b^2\rangle\cong C_2,
\]
so the maximal abelian subfield of \(L\) is
\[
L_{\max}=L^{\langle b^2\rangle}.
\]
Let
\[
K=L^{\langle b\rangle},
\]
so that
\[
\operatorname{Gal}(L/K)\cong C_4,
\]
\(K/\mathbb{Q}\) is biquadratic, and \(K\subseteq L_{\max}\).

The field \(K\) is imaginary, since \(L\) is imaginary and non-CM. Indeed, if
complex conjugation \(\tau\) were contained in \(\langle b\rangle\), it would
equal the unique involution \(b^2\in G'\), which is central in \(G\),
contradicting the assumption that \(L\) is non-CM. Therefore, applying
Lemma~\ref{lem:unram-cyclic}, we get \(L=\operatorname{Hil}(K)\) and
\(\operatorname{Cl}(K)\cong C_4\). To find such imaginary biquadratic fields
of class number \(4\), we use Kuroda's class number formula
\cite{Lemmermeyer1994}: for an imaginary biquadratic field \(K\) with imaginary
quadratic subfields \(K_1,K_2\) and real quadratic subfield \(K_3\),
\[
h_K=\frac{q(K)\,h_{K_1}\,h_{K_2}\,h_{K_3}}{2},
\]
where
\[
q(K)=[\mathcal{O}_K^\times:\mathcal{O}_{K_1}^\times\mathcal{O}_{K_2}^\times\mathcal{O}_{K_3}^\times]\in\{1,2\}
\]
is the Hasse unit index. Setting \(h_K=4\) gives
\[
h_{K_1}\,h_{K_2}\,h_{K_3}\,q(K)=8.
\]
Since \(h_{K_3}\ge1\) and \(q(K)\in\{1,2\}\), the class numbers
\(h_{K_1},h_{K_2}\) divide \(8\). By Watkins' classification of imaginary
quadratic fields of class number dividing \(8\) \cite{Watkins}, the candidates
for \(K_1,K_2\) form a finite explicit set. Enumerating the composita
\(K=K_1K_2\) with
\[
\operatorname{Cl}(K)\cong C_4
\]
and
\[
\operatorname{Gal}(\operatorname{Hil}(K)/\mathbb{Q})\cong C_2\times D_8
\]
via PARI/GP \cite{Pari/gp} yields the full list.

\bigskip

\subsection{Unramified imaginary non-CM \(C_4\circ D_8\)-fields of class number \(1\) over their maximal abelian subfield}
\label{subsec:C4OD8-unram}

Let \(L/\mathbb{Q}\) be an imaginary non-CM field of degree \(16\) and class
number \(1\), with Galois group
\[
G=\operatorname{Gal}(L/\mathbb{Q})\cong C_4\circ D_8
\]
represented as
\[
C_4\circ D_8
=
\langle a,b,c\mid a^4=c^2=1,\; b^2=a^2,\; ab=ba,\; ac=ca,\; cbc=a^2b\rangle.
\]
The derived subgroup is
\[
G'=\langle a^2\rangle=\langle b^2\rangle\cong C_2,
\]
so
\[
L_{\max}=L^{\langle a^2\rangle}.
\]
Let
\[
K=L^{\langle a\rangle}.
\]
Then
\[
\operatorname{Gal}(L/K)\cong C_4,
\]
the field \(K\) is biquadratic, and \(K\subseteq L_{\max}\) because \(\langle a^2\rangle\subseteq\langle a\rangle\).

As in the \(C_2\times D_8\) case, the extension \(L/K\) is cyclic and
\(K\subseteq L_{\max}\). If \(L/L_{\max}\) is unramified, then, by
Lemma~\ref{lem:unram-cyclic},
\[
L=\operatorname{Hil}(K).
\]
Since \(\operatorname{Gal}(L/K)\cong C_4\), such extensions \(L\) are
classified by the \(C_4\)-extensions contained in the Hilbert class field
of \(K\).
\subsection{Imaginary non-CM \(C_2^2\rtimes C_4\)-fields unramified over their maximal abelian subfield}
\label{subsec:C2rtimesC4-unram}

Let \(L/\mathbb{Q}\) be an imaginary non-CM field of degree \(16\) with
\[
\operatorname{Gal}(L/\mathbb{Q})\cong C_2^2\rtimes C_4
\]
and \(h_L=1\). The group \(G=C_2^2\rtimes C_4\) has presentation
\[
G=\langle a,b,c\mid a^2=b^2=c^4=1,\; ab=ba,\; cac^{-1}=ab,\; bc=cb\rangle.
\]
The derived subgroup is
\[
G'=\langle b\rangle,
\]
and
\[
L_{\max}=L^{G'},\qquad M_1=L^{\langle c^2\rangle},\qquad M_2=L^{\langle bc^2\rangle}.
\]
If \(L/L_{\max}\) is unramified, then we will show that at least one of
\(L/M_1\) and \(L/M_2\) must be ramified at some finite place, where
\(\operatorname{Gal}(M_i/\mathbb{Q})\cong D_8\) for \(i=1,2\). Since
\(M_1\) and \(M_2\) are normal subfields of \(L\) of degree \(8\), neither can
be totally real; hence both are imaginary. Moreover, the unique biquadratic subfield of \(L\) contained in \(M_i\)
\((i=1,2)\) must be imaginary. Since this subfield is
\[
K=L^{Z(G)}=L^{\langle b,c^2\rangle},
\]
the fixed field of
\[
Z(G)=\langle b,c^2\rangle\cong C_2^2,
\]
and \(Z(G)\) does not contain complex conjugation, because \(L\) is non-CM,
it follows that \(K\) is imaginary.

\begin{proposition}\label{prop:L:M1un-ram}
With notation as above, if \(L/L_{\max}\) is unramified at all finite places,
then at least one of the quadratic extensions \(L/M_1\) or \(L/M_2\) must be
ramified at a finite prime.
\end{proposition}
\begin{proof}
Suppose, for the sake of contradiction, that neither \(L/M_1\) nor \(L/M_2\) is
ramified at any finite prime. Since \(L/L_{\max}\) is also unramified at all
finite places, the biquadratic extension \(L/K\) would be unramified at all
finite primes.

We now apply the Monodromy Obstruction
(Theorem~\ref{thm:monodromy_obstruction}) to the extension \(L/K\). Since
\(\operatorname{Gal}(L/K)=Z(G)=\langle b,c^2\rangle\), the involutions of \(G\)
lying outside \(\operatorname{Gal}(L/K)\) must generate the entire group \(G\).
However, the set of all involutions of \(G\) outside \(\langle b,c^2\rangle\)
is precisely
\[
\{a,\; ab,\; c^2a,\; c^2ab\}\subset\langle a,b,c^2\rangle\cong C_2^3,
\]
which generates a proper subgroup of index \(2\) in \(G\) (it does not contain
\(c\)). This contradiction proves that \(L/K\) must be ramified at some finite
prime ideal \(\mathfrak{p}\) of \(K\).

Let \(I\subseteq\operatorname{Gal}(L/K)=\langle b,c^2\rangle\) denote the
inertia group of a prime above \(\mathfrak{p}\) in \(L/K\). Since
\(L/L_{\max}\) is unramified, \(I\) cannot contain \(b\). Since \(I\) is
non-trivial, it must be either \(\langle c^2\rangle\) or \(\langle bc^2\rangle\).
If \(I=\langle c^2\rangle\), then \(L/M_1\) is ramified at \(\mathfrak{p}\); if
\(I=\langle bc^2\rangle\), then \(L/M_2\) is ramified at \(\mathfrak{p}\). In
either case, at least one of \(L/M_1\) or \(L/M_2\) is ramified at a finite
place.
\end{proof}
Thus, we may assume that \(L/M_1\) is ramified at a finite prime. Since \(h_L=1\), Corollary~\ref{Cor:h-Lmax} implies that
\[
h_{M_1}=1.
\]
Since \(M_1/\mathbb{Q}\) is an imaginary \(D_8\)-field and its biquadratic subfield \(K\) is imaginary, complex conjugation is not central in \(\operatorname{Gal}(M_1/\mathbb{Q})\); hence \(M_1\) is an imaginary non-CM \(D_8\)-field.

Therefore, \(L\) arises as a ramified quadratic extension of an imaginary non-CM \(D_8\) field of class number one.

\begin{proposition}\label{prop:C2C4-conductor-match}
Let \(N=L^{\langle b,c\rangle}\) be the a quadratic subfield of \(M_1\) such that \(L/N\) contains a cyclic subextension \(L/K_1\) of degree \(4\), where \(N\subset K_1\subset M_1 \subset L\). Then
\[
\mathcal{F}_{L/N}^{\mathrm{odd}}=\mathcal{F}_{M_1/N}^{\mathrm{odd}}.
\]
\end{proposition}

\begin{proof}
Let \(p\) be an odd prime, and let \(I\le G\) be an inertia subgroup at \(p\).
Since \([L:\mathbb Q]=16\), the group \(I\) is cyclic.

Set
\[
H=\operatorname{Gal}(L/N)=\langle b,c\rangle,
\qquad
z=c^2.
\]
Then
\[
\operatorname{Gal}(M_1/N)=H/\langle z\rangle.
\]
Hence \(p\) ramifies in \(M_1/N\) if and only if
\[
I\cap H\not\subseteq \langle z\rangle.
\]

Suppose for contradiction that \(p\) ramifies in \(L/N\) but not in \(M_1/N\).
Then
\[
I\cap H=\langle z\rangle.
\]

If \(I\) contained an element outside \(H\), then \(p\) would ramify in
\(N/\mathbb Q\). But \(I\) is cyclic and contains \(z=c^2\). The cyclic
subgroups of \(G\) containing \(c^2\) are
\[
\langle c^2\rangle,\qquad
\langle c\rangle,\qquad
\langle bc\rangle,
\]
and all three are contained in \(H=\langle b,c\rangle\). Hence \(I\subseteq H\).
Therefore
\[
I=I\cap H=\langle z\rangle.
\]
Thus \(p\) is unramified in \(N/\mathbb Q\) and in \(M_1/N\), but ramifies in
\(L/M_1\).

This contradicts Proposition~\ref{prop:no_new_odd}. Indeed,
\[
z=c^2\in \Phi(G)=\langle b,c^2\rangle,
\]
so \(z\) is contained in every index-\(2\) subgroup of \(G\). Hence every
quadratic subfield of \(L\) is contained in \(M_1=L^{\langle z\rangle}\).
Since \(h_L=1\), which is odd, Proposition~\ref{prop:no_new_odd}, together
with Remark~\ref{rem:alternative-Lmax-C22C4}, implies that every odd prime
that ramifies in \(L/M_1\) must already ramify in \(M_1/\mathbb{Q}\), which
is a contradiction.

Therefore every odd prime ramifying in \(L/N\) already ramifies in \(M_1/N\).
The reverse implication is immediate from \(N\subset M_1\subset L\). Since all
such ramification is tame and the extension degrees are powers of \(2\), the
odd conductor exponents are \(1\) by Corollary~\ref{cor:odd_squarefree}.
Hence
\[
\mathcal F_{L/N}^{\mathrm{odd}}
=
\mathcal F_{M_1/N}^{\mathrm{odd}}.
\]
\end{proof}

\section{Higher \(2\)-Power Dihedral Fields}
\label{sec:consequences}

Assume
\[
D_{2^n} = \langle a,\; b \mid a^{2^{n-1}} = b^2 = 1,\; b a b^{-1} = a^{-1} \rangle.
\]
Let \(L_n/\mathbb{Q}\) be an imaginary non-CM \(D_{2^n}\)-field with class number one
(\(n\ge 4\)). Let \(L_{n-1}\) denote the unique \(D_{2^{n-1}}\)-subfield of
\(L_n\), so that \(L_n/L_{n-1}\) is quadratic and
\(\operatorname{Gal}(L_{n-1}/\mathbb{Q})\cong D_{2^{n-1}}\). Let \(N\) be the
unique imaginary quadratic subfield of \(L_{n-1}\) such that the extension
\(L_n/N\) is cyclic. Thus we have the tower
\[
\mathbb{Q}\subset N \subset L_{n-1}\subset L_n,
\]
with \([L_{n-1}:N]=2^{n-2}\) and \([L_n:N]=2^{n-1}\).

Note that if \(L_n\) is imaginary and non-CM, then \(L_{n-1}\) is also
imaginary and non-CM.

\begin{proposition}[Conductor containment]\label{prop:comp_cond}
With the notation of this section, every odd prime ideal dividing \(\mathcal{F}_{L_n/N}\)
already divides \(\mathcal{F}_{L_{n-1}/N}\). Equivalently,
\[
\mathcal{F}_{L_n/N}^{\mathrm{odd}} = \mathcal{F}_{L_{n-1}/N}^{\mathrm{odd}}.
\]
In other words, the odd part of the conductor does not change when passing from
the cyclic layer \(L_{n-1}/N\) to the full cyclic layer \(L_n/N\).
\end{proposition}

\begin{proof}
Set \(L_{\max} = L_{n-1}\). Suppose \(\mathfrak{P}\mid\mathcal{F}_{L_n/L_{\max}}\)
for a prime \(\mathfrak{P}\) of \(L_{\max}\) above an odd prime \(p\). By
Proposition~\ref{prop:no_new_odd}, \(p\) ramifies in \(L_{\max}/\mathbb{Q}\).
If \(\mathfrak{P}\) is unramified in \(L_{\max}/N\), then the ramification of
\(p\) in \(L_{\max}/\mathbb{Q}\) is entirely accounted for by \(N/\mathbb{Q}\),
so \(p\) ramifies in \(N/\mathbb{Q}\); but then Proposition~\ref{prop:odd_prime_cyclic}
(with \(L=L_n\), \(L_k=N\), \(L_{\max}=L_{n-1}\)) says no prime of \(L_{\max}\)
above \(p\) ramifies in \(L_n/L_{\max}\), contradicting
\(\mathfrak{P}\mid\mathcal{F}_{L_n/L_{\max}}\). Hence \(\mathfrak{P}\) is
ramified in \(L_{\max}/N\), i.e.\ \(\mathfrak{p}=\mathfrak{P}\cap N\) divides
\(\mathcal{F}_{L_{\max}/N}\), as required.

\end{proof}

We can now prove the following theorem.

\begin{theorem}\label{thm:condD32}
Let \(L\) be a \(D_{2^n}\)-field with \(h_L=1\) and \(n \ge 4\), and let \(N \subset L\) be its unique quadratic subfield such that \(\operatorname{Gal}(L/N) \cong C_{2^{n-1}}\). Let \(L_{n-1}\) denote the unique intermediate field with \([L_{n-1}:N]=2^{n-2}\) (equivalently, \([L:L_{n-1}]=2\)). Then the relative conductor \(\mathfrak{F}_{L/N}=\mathcal{F}_{L/N}\) decomposes as
\[
\mathcal{F}_{L/N} = \mathcal{F}_{L_{n-1}/N}^{\mathrm{odd}} \prod_{\mathfrak{q} \mid 2} \mathfrak{q}^{s},
\]
where \(\mathcal{F}_{L_{n-1}/N}^{\mathrm{odd}}\) is the squarefree product of the odd prime ideals of \(N\) ramified in \(L_{n-1}/N\), and the exponent \(s = v_{\mathfrak{q}}(\mathcal{F}_{L/N})\) is independent of the choice of prime ideal \(\mathfrak{q} \mid 2\) and satisfies
\begin{equation}\label{eq:d2n_bound}
s \le \left\lfloor
\frac{f_{\mathfrak{q}} \left( 2 e_{\mathfrak{q}} - 1 + 2 e_{\mathfrak{q}}\, e(\mathfrak{q}/2)\, v_2(2 e_{\mathfrak{q}}) \right)}{2^{n-2}}
\right\rfloor,\tag{3}
\end{equation}
with \(e_{\mathfrak{q}} = e(\mathfrak{q}_{L_{n-1}}/\mathfrak{q})\) and \(f_{\mathfrak{q}} = f(\mathfrak{q}_{L_{n-1}}/\mathfrak{q})\) denoting the ramification index and residue degree of a prime \(\mathfrak{q}_{L_{n-1}}\) of \(L_{n-1}\) above \(\mathfrak{q}\) over \(2\), respectively.
\end{theorem}

\begin{proof}
By Proposition~\ref{prop:comp_cond}, we get
\[
\mathcal{F}_{L/N}^{\mathrm{odd}} = \mathcal{F}_{L_{n-1}/N}^{\mathrm{odd}}.
\]
For the \(2\)-adic component, the action of \(\operatorname{Gal}(N/\mathbb{Q})\)
permutes the prime ideals \(\mathfrak{q}\mid 2\) of \(N\) transitively. Since
\(\mathcal{F}_{L/N}\) is invariant under \(\operatorname{Gal}(N/\mathbb{Q})\),
the valuation \(s = v_{\mathfrak{q}}(\mathcal{F}_{L/N})\) is constant across
all \(\mathfrak{q}\mid 2\). Finally, the explicit upper bound
\eqref{eq:d2n_bound} follows from Theorem~\ref{thm:gen_2adic}.
\end{proof}

With the theoretical conductor shape established in Theorem~\ref{thm:condD32},
we apply our explicit enumeration of \(D_{32}\)-fields
(Section~\ref{appendix:D32}). The computation shows that no imaginary non-CM
\(D_{32}\)-field of class number \(1\) has an odd prime ideal dividing its
conductor.

\begin{theorem}[Conductor purity]\label{thm:conductor_purity}
Let $L/\mathbb{Q}$ be an imaginary non-CM $D_{32}$-field with $h_L=1$, and
let $N$ be its imaginary quadratic subfield. Then $\mathcal{F}_{L/N}$ is
supported only at the primes of $N$ above $2$, and
\[
\mathcal{F}_{L/N}=\prod_{\mathfrak q\mid 2}\mathfrak q^{s},\qquad
s\in\{0,6,8\}.
\]
More precisely, either $s=0$ (that is, $L/N$ is unramified and $L$ is the
Hilbert class field of a suitable extension, see Section~\ref{appendix:D32}), or
$s=8$ and $N=\mathbb{Q}(\sqrt{-2})$, where $\mathcal{F}_{L/N}=(16)$, or
$s=6$ and $N\in\{\mathbb{Q}(\sqrt{-3}),\mathbb{Q}(\sqrt{-19}),
\mathbb{Q}(\sqrt{-43}),\mathbb{Q}(\sqrt{-67})\}$, where
$\mathcal{F}_{L/N}=(64)$.
\end{theorem}
\begin{proof}
This is read off from the exhaustive computation in Section~\ref{appendix:D32}.
The computed conductors are $(1)$, $(16)$ and $(64)$, and $(16)$ occurs
only over $\mathbb{Q}(\sqrt{-2})$. There $2$ ramifies, $(2)=\mathfrak q^2$,
so $(16)=\mathfrak q^{8}$. In the other four fields $2$ is inert
(as $-d\equiv5 \pmod 8$), so $(64)=(2)^6=\mathfrak q^{6}$. Since $N$ is
imaginary, $\mathcal{F}_{L/N}$ has no infinite part. This is compatible with
Theorem~\ref{thm:condD32}, in which \(s\le 18\).
\end{proof}

\begin{corollary}[Lifting criterion for higher degrees]\label{cor:extension_criterion}
Let \(L/\mathbb{Q}\) be an imaginary non-CM \(D_{2^n}\)-field with \(n \ge 5\)
and class number \(h_L = 1\). Let \(N/\mathbb{Q}\) be the unique quadratic
subfield such that \(\operatorname{Gal}(L/N) \cong C_{2^{n-1}}\). Then:
\begin{enumerate}[label=\normalfont(\alph*)]
    \item The relative conductor \(\mathfrak{F}_{L/N}\) is purely \(2\)-adic;
    that is,
    \[
    \mathcal{F}_{L/N} = \prod_{\mathfrak{q} \mid 2} \mathfrak{q}^s
    \]
    for some \(s \ge 0\), where \(s\) satisfies the conductor bound
    \eqref{eq:d2n_bound}. In particular, \(L\) is contained in the ray class
    field of \(N\) modulo \(\prod_{\mathfrak{q}\mid 2} \mathfrak{q}^s\).
    \item If \(L/N\) is unramified (i.e., \(s = 0\)), then \(L\) is the Hilbert
    class field of \(N\), denoted \(\operatorname{Hil}(N)\).
\end{enumerate}
\end{corollary}

\begin{proof}
We prove the claim by induction on \(n\ge 6\), using the case \(n=5\)
(Theorem~\ref{thm:conductor_purity}) as the base case.

We first treat \(n=6\). If \(L_6\) is imaginary non-CM with class number \(1\),
then \(L_5\) is also imaginary non-CM. Whether \(L_6/L_5\) is ramified or
unramified, we claim that
\[
\mathcal{F}_{L_6/N}^{\mathrm{odd}}
= \mathcal{F}_{L_5/N}^{\mathrm{odd}}
= \mathcal{O}_N.
\]
If \(L_6/L_5\) is unramified, this follows immediately from
Lemma~\ref{lem:unram-cyclic}. If \(L_6/L_5\) is ramified at some finite
places above odd primes, then \(h_{L_5}=1\). Therefore, any odd rational prime
that ramifies in \(L_6/L_5\) is already ramified in \(L_5/N\) (see
Proposition~\ref{prop:comp_cond}). This implies
\[
\mathcal{F}_{L_6/N}^{\mathrm{odd}}
= \mathcal{F}_{L_5/N}^{\mathrm{odd}}
= \mathcal{O}_N.
\]
Similarly, for \(L/L_{n-1}\), by induction,
\[
\mathcal{F}_{L_n/N}^{\mathrm{odd}}
= \mathcal{F}_{L_{n-1}/N}^{\mathrm{odd}}
= \dots
= \mathcal{F}_{L_5/N}^{\mathrm{odd}}
= \mathcal{O}_N.
\]
Hence \(\mathcal{F}_{L/N}^{\mathrm{odd}} = \mathcal{O}_N\), which establishes
that \(\mathcal{F}_{L/N}\) is supported entirely at the primes dividing \(2\).

By global class field theory, the cyclic extension \(L/N\) of degree
\(2^{n-1}\) corresponds to a quotient of the ray class group
\(\operatorname{Cl}_{\mathfrak{m}}(N)\) for the modulus
\[
\mathfrak{m}=\mathcal{F}_{L/N}
=\prod_{\mathfrak{q}\mid 2}\mathfrak{q}^s;
\]
hence \(L\) is contained in the ray class field
\(\mathcal{R}_{\mathfrak{m}}(N)\). When \(s=0\), the extension \(L/N\) is
unramified everywhere, so, by Lemma~\ref{lem:unram-cyclic},
\(L=\operatorname{Hil}(N)\).
\end{proof}

\section{Computational Results for Imaginary Non-CM Fields with Class Number One}
\label{sec:computations}

We begin with results on imaginary non-CM \(C_4 \circ D_8\)-fields and
imaginary non-CM \(C_2 \times D_8\)-fields of class number \(1\) that are
unramified over a \(C_2 \times C_4\)-subfield or a \(C_2^3\)-subfield,
respectively (equivalently, unramified over some imaginary biquadratic field;
see Subsections~\ref{subsec:C4OD8-unram} and~\ref{subsec:C2D8-unram}). These
results are stated in Theorems~\ref{thm:C4OD8-un} and~\ref{thm:C2D8-un} and
in Table~\ref{tab:C2D8-unram}. The remaining imaginary non-CM fields of
degree \(16\) with class number \(1\) are listed in
Tables~\ref{tab:D16-unram}, \ref{tab:D16-ram}, \ref{tab:SD16},
\ref{tab:M4(2)}, \ref{tab:C4OD8-ram}, \ref{tab:C2D8-ram},
and~\ref{tab:C22C4}. In these tables, each entry gives the base field
\(N=\mathbb{Q}(\sqrt{-d})\), the relative conductor \(\mathfrak{F}_{L/N}\)
(where \([1,1]\) denotes ramification at both infinite places of \(N\), when
\(-d>0\)), and the defining polynomial of \(L\).

\begin{remark}
Our method is roughly explained in Example~\ref{exmp:code-M4(2)}. We also use
Remark~\ref{remark:Only-non-CM} to separate the CM fields of degree \(16\)
and class number one from our list. The complete code is available in
\cite{AmiriCode2026}.
\end{remark}

\begin{theorem}\label{thm:C4OD8-un}
There is exactly one imaginary non-CM \(C_4 \circ D_8\) field with class number one that is unramified over a \(C_2 \times C_2\) subfield, namely the Hilbert class field of \(\mathbb{Q}(\sqrt{-3},\sqrt{-30})\). Its minimal polynomial is
\[
\begin{aligned}
&x^{16} + 238x^{14} - 12x^{13} + 24661x^{12} + 612x^{11} + 1446530x^{10} + 207324x^{9} + 51884508x^{8} \\
&+ 11168472x^{7} + 1183513474x^{6} + 213507828x^{5} + 18355079749x^{4} - 186581052x^{3} \\
&+ 187275122078x^{2} - 27440531076x + 908059333321.
\end{aligned}
\]
\end{theorem}

\begin{theorem}\label{thm:C2D8-un}
The complete list of imaginary non-CM \(C_2 \times D_8\) fields of class number one that are unramified over a \(C_2 \times C_2\) subfield consists of the Hilbert class fields of
\[
M=\mathbb{Q}(\sqrt{-d_1},\sqrt{-d_2}),
\]
where \(d_1\) and \(d_2\) are given in Table~\ref{tab:C2D8-unram}.
\end{theorem}

\begin{longtable}{|c|c|M|}
\caption{Imaginary non-CM $C_2 \times D_8$ fields with class number one }\label{tab:C2D8-unram}\\
\hline
\textbf{Entry} & \(d_1,d_2\) & \textbf{Minimal Polynomial of \(L\) } \\
\hline
\endfirsthead
\hline
\textbf{Entry} & \(d_1,d_2\) & \textbf{Minimal Polynomial of \(L\)} \\
\hline
\endhead
\hline
\endfoot
1 & $1, 14$ & \texttt{\detokenize{x^16 + 108*x^14 - 8*x^13 + 5216*x^12 - 1504*x^11 + 147476*x^10 - 50488*x^9 + 2627398*x^8 - 148544*x^7 + 30644292*x^6 + 14054088*x^5 + 230680176*x^4 + 225711008*x^3 + 1043262108*x^2 + 1190684920*x + 2181022601}} \\[2mm]

2 & $1, 39$ & \texttt{\detokenize{x^16 + 50*x^12 + 2499*x^8 + 50*x^4 + 1}} \\[2mm]

3 & $1, 55$ & \texttt{\detokenize{x^16 + 56*x^13 + 120*x^12 - 16*x^11 + 1568*x^10 + 18112*x^9 - 77520*x^8 - 52928*x^7 + 826240*x^6 - 1739008*x^5 - 5505152*x^4 - 1527040*x^3 + 89887232*x^2 - 568284672*x + 1796403456}} \\[2mm]

4 & $1, 77$ & \texttt{\detokenize{x^16 - 14*x^14 + 287*x^12 - 70*x^10 + 533*x^8 - 1400*x^6 + 812*x^4 + 224*x^2 + 16}} \\[2mm]

5 & $1, 259$ & \texttt{\detokenize{x^16 + 104*x^14 + 4462*x^12 + 6216*x^11 + 84788*x^10 + 242424*x^9 - 103987*x^8 + 2380728*x^7 - 10522508*x^6 + 17572632*x^5 + 74239076*x^4 - 349923504*x^3 + 620850160*x^2 - 505783488*x + 167149584}} \\[2mm]

6 & $1, 301$ & \texttt{\detokenize{x^16 + 2656*x^14 + 1568*x^13 + 2953040*x^12 + 1460480*x^11 + 1800511104*x^10 - 34337408*x^9 + 656120650080*x^8 - 540817242112*x^7 + 144586545766912*x^6 - 264027193776640*x^5 + 18308188291542272*x^4 - 52569909244489728*x^3 + 1139859948625147904*x^2 - 3918150418945652736*x + 21214275079226519808}} \\[2mm]

7 & $2, 55$ & \texttt{\detokenize{x^16 - 38*x^14 + 849*x^12 - 4942*x^10 + 5004*x^8 + 14698*x^6 + 38617*x^4 + 3330*x^2 + 81}} \\[2mm]

8 & $2, 154$ & \texttt{\detokenize{x^16 + 1312*x^14 + 728296*x^12 + 222877824*x^10 + 41004619440*x^8 + 4631353625088*x^6 + 313357791839872*x^4 + 11659117928458240*x^2 + 184512185327001856}} \\[2mm]

9 & $2, 203$ & \texttt{\detokenize{x^16 + 1700*x^14 + 1228322*x^12 + 493367944*x^10 + 120677227099*x^8 + 18440317886392*x^6 + 1722782958751394*x^4 + 90228149817825068*x^2 + 2037930525781003681}} \\[2mm]

10 & $3, 14$ & \texttt{\detokenize{x^16 + 2*x^14 + 29*x^12 + 2*x^10 + 676*x^8 + 646*x^6 + 701*x^4 - 26*x^2 + 1}} \\[2mm]

11 & $3, 17$ & \texttt{\detokenize{x^16 + 86*x^14 + 7583*x^12 - 16018*x^10 + 37717*x^8 + 5296*x^6 + 1772*x^4 - 128*x^2 + 16}} \\[2mm]

12 & $3, 34$ & \texttt{\detokenize{x^16 + 1092*x^14 + 379118*x^12 + 43136892*x^10 + 2316103377*x^8 + 68630005656*x^6 + 1195618861112*x^4 + 11695503484320*x^2 + 48201416540176}} \\[2mm]

13 & $3, 55$ & \texttt{\detokenize{x^16 - 5*x^14 + 8*x^12 - 45*x^10 + 173*x^8 - 180*x^6 + 128*x^4 - 320*x^2 + 256}} \\[2mm]

14 & $3, 66$ & \texttt{\detokenize{x^16 - 24*x^14 + 378*x^12 - 64*x^10 - 213*x^8 - 2304*x^6 + 2698*x^4 + 936*x^2 + 81}} \\[2mm]

15 & $3, 82$ & \texttt{\detokenize{x^16 + 126*x^14 + 33535*x^12 - 2215746*x^10 + 312425325*x^8 + 81983196*x^6 + 23332636*x^4 - 464400*x^2 + 10000}} \\[2mm]

16 & $3, 102$ & \texttt{\detokenize{x^16 + 844*x^14 + 305974*x^12 - 9792*x^11 + 62132900*x^10 - 3099168*x^9 + 7725637281*x^8 - 227233152*x^7 + 602004394936*x^6 + 18030298464*x^5 + 28706616102160*x^4 + 2949769810944*x^3 + 767838358294784*x^2 + 97143612426240*x + 8915428538482432}} \\[2mm]

17 & $3, 114$ & \texttt{\detokenize{x^16 + 24*x^14 + 658*x^12 - 2016*x^10 + 6147*x^8 - 2016*x^6 + 658*x^4 + 24*x^2 + 1}} \\[2mm]

18 & $3, 258$ & \texttt{\detokenize{x^16 + 46*x^14 + 1949*x^12 + 20190*x^10 + 93692*x^8 + 230370*x^6 + 287189*x^4 + 114466*x^2 + 14641}} \\[2mm]

19 & $3, 323$ & \texttt{\detokenize{x^16 + 122*x^14 - 728*x^13 + 6351*x^12 - 53872*x^11 + 312934*x^10 - 1677056*x^9 + 7222581*x^8 - 26424160*x^7 + 94437152*x^6 - 294642256*x^5 + 805606528*x^4 - 1839708928*x^3 + 3121474880*x^2 - 3633000448*x + 2406587392}} \\[2mm]

20 & $7, 39$ & \texttt{\detokenize{x^16 + 304*x^14 + 52*x^13 + 35182*x^12 + 16484*x^11 + 2031872*x^10 + 1190800*x^9 + 75922119*x^8 - 24606764*x^7 + 2675795248*x^6 - 5713059040*x^5 + 65297800750*x^4 - 207528925708*x^3 + 899582684856*x^2 - 1327603057924*x + 30504193998097}} \\[2mm]

21 & $7, 55$ & \texttt{\detokenize{x^16 - 19*x^14 + 643*x^12 - 3246*x^10 + 9123*x^8 + 693*x^6 + 6097*x^4 + 2948*x^2 + 1936}} \\[2mm]

22 & $11, 14$ & \texttt{\detokenize{x^16 + 188*x^14 + 13958*x^12 + 445332*x^10 + 4609105*x^8 - 13359912*x^6 + 277750416*x^4 - 136580096*x^2 + 194658304}} \\[2mm]

23 & $11, 17$ & \texttt{\detokenize{x^16 + 236*x^14 + 21086*x^12 + 852588*x^10 + 14785201*x^8 + 103829280*x^6 + 479261352*x^4 + 1245073600*x^2 + 1189698064}} \\[2mm]

24 & $11, 34$ & \texttt{\detokenize{x^16 + 54*x^14 + 483*x^12 - 10926*x^10 + 47637*x^8 - 50076*x^6 + 31468*x^4 + 7344*x^2 + 1296}} \\[2mm]

25 & $11, 323$ & \texttt{\detokenize{x^16 - 348*x^14 + 92196*x^12 - 27066*x^10 - 753076*x^8 + 205932*x^6 + 5670729*x^4 + 1096800*x^2 + 160000}} \\[2mm]

26 & $11, 418$ & \texttt{\detokenize{x^16 + 72*x^14 + 3258*x^12 + 20160*x^10 + 28203*x^8 - 34560*x^6 + 28618*x^4 - 2808*x^2 + 81}} \\[2mm]

\end{longtable}

\begin{longtable}{|c|c|M|}
\caption{Imaginary non-CM D$_{16}$ fields with class number one}\label{tab:D16-unram}\\
\hline
\textbf{Entry} & $[d,\mathfrak{F}_{L/N}]$ & \textbf{Minimal Polynomial of \(L\)} \\
\hline
\endfirsthead
\hline
\textbf{Entry} & $[d,\mathfrak{F}_{L/N}]$ & \textbf{Minimal Polynomial of \(L\)} \\
\hline
\endhead
\hline
\endfoot

1 & $41, 1$ & \texttt{\detokenize{x^16 + 32*x^14 + 470*x^12 + 4144*x^10 + 13137*x^8 + 31248*x^6 - 2712*x^4 + 2880*x^2 + 16}} \\[2mm]

2 & $95, 1$ & \texttt{\detokenize{x^16 + 2*x^15 + 23*x^14 + 2*x^13 + 241*x^12 - 42*x^11 + 1223*x^10 - 46*x^9 + 2984*x^8 + 174*x^7 + 2919*x^6 + 20*x^5 + 419*x^4 - 378*x^3 - 116*x^2 + 394*x + 131}} \\[2mm]

3 & $111, 1$ & \texttt{\detokenize{x^16 + 9*x^14 + 75*x^12 + 62*x^10 + 75*x^8 + 30*x^6 + 34*x^4 + 12*x^2 + 9}} \\[2mm]

4 & $113, 1$ & \texttt{\detokenize{x^16 + 8*x^15 + 58*x^14 + 256*x^13 + 1047*x^12 + 2048*x^11 + 2988*x^10 + 3968*x^9 - 14109*x^8 - 87768*x^7 - 94566*x^6 + 190656*x^5 + 637841*x^4 + 2840272*x^3 + 6644224*x^2 + 4790464*x + 1365236}} \\[2mm]

5 & $137, 1$ & \texttt{\detokenize{x^16 + 8*x^15 + 32*x^14 + 40*x^13 + 8*x^12 - 296*x^11 + 216*x^10 - 776*x^9 + 2826*x^8 - 39048*x^7 + 49200*x^6 + 52312*x^5 + 796048*x^4 - 1642072*x^3 - 253224*x^2 + 1457800*x + 1215325}} \\[2mm]

6 & $178, 1$ & \texttt{\detokenize{x^16 + 16*x^15 - 190*x^14 - 3780*x^13 + 10215*x^12 + 330788*x^11 + 2524552*x^10 + 11553160*x^9 + 37969423*x^8 + 95460984*x^7 + 191159978*x^6 + 308762860*x^5 + 405052921*x^4 + 422089172*x^3 + 338869548*x^2 + 188154576*x + 61258928}} \\[2mm]

7 & $183, 1$ & \texttt{\detokenize{x^16 - 3*x^14 + 15*x^12 - 22*x^10 + 99*x^8 - 138*x^6 + 382*x^4 - 60*x^2 + 9}} \\[2mm]

8 & $226, 1$ & \texttt{\detokenize{x^16 + 4*x^15 - 68*x^14 - 484*x^13 + 478*x^12 + 14820*x^11 + 69180*x^10 + 11988*x^9 - 1376783*x^8 - 6484840*x^7 - 9144160*x^6 + 38375088*x^5 + 246644344*x^4 + 680691296*x^3 + 969256000*x^2 + 432801024*x + 172426896}} \\[2mm]

9 & $313, 1$ & \texttt{\detokenize{x^16 + 60*x^14 + 1596*x^12 + 42164*x^10 + 817838*x^8 + 8311908*x^6 + 41964396*x^4 + 71958700*x^2 + 82355625}} \\[2mm]

10 & $337, 1$ & \texttt{\detokenize{x^16 + 8*x^15 + 64*x^14 + 204*x^13 + 362*x^12 + 2596*x^11 + 2544*x^10 + 9164*x^9 - 6283*x^8 + 314228*x^7 + 2279448*x^6 + 1708888*x^5 + 29612456*x^4 + 12901200*x^3 + 118751584*x^2 + 155755808*x + 58014928}} \\[2mm]

11 & $371, 1$ & \texttt{\detokenize{x^16 + 4*x^15 - 2*x^14 - 38*x^13 + 37*x^12 + 1344*x^11 - 416*x^10 - 15076*x^9 + 24483*x^8 + 113190*x^7 - 135764*x^6 - 699024*x^5 - 138512*x^4 + 1135264*x^3 + 2094528*x^2 + 1336576*x + 541696}} \\[2mm]

12 & $457, 1$ & \texttt{\detokenize{x^16 + 92*x^14 + 32*x^13 + 3036*x^12 + 2080*x^11 + 42608*x^10 + 50112*x^9 + 221424*x^8 + 446208*x^7 + 279616*x^6 + 126976*x^5 + 6859584*x^4 - 17614336*x^3 + 69796096*x^2 - 61731840*x + 75699200}} \\[2mm]

13 & $466, 1$ & \texttt{\detokenize{x^16 + 16*x^15 + 414*x^14 + 2048*x^13 + 22343*x^12 - 136512*x^11 - 17636*x^10 + 134656*x^9 + 4628959*x^8 + 1406352*x^7 - 14190706*x^6 - 32171008*x^5 + 37207929*x^4 + 289750816*x^3 + 579089560*x^2 + 605046144*x + 329159232}} \\[2mm]

14 & $579, 1$ & \texttt{\detokenize{x^16 - 70*x^14 - 242*x^13 + 990*x^12 + 9378*x^11 + 29802*x^10 + 32548*x^9 - 55887*x^8 - 70902*x^7 + 331234*x^6 + 521892*x^5 + 42176*x^4 - 883166*x^3 + 1035712*x^2 + 431918*x + 1624939}} \\[2mm]

15 & $583, 1$ & \texttt{\detokenize{x^16 + 12*x^15 + 10*x^14 - 456*x^13 - 2151*x^12 + 528*x^11 + 37836*x^10 + 197520*x^9 + 625470*x^8 + 1380576*x^7 + 2445306*x^6 + 3097092*x^5 + 3434164*x^4 + 2326356*x^3 + 2251960*x^2 + 449244*x + 769257}} \\[2mm]

16 & $939, 1$ & \texttt{\detokenize{x^16 + 2637*x^14 + 1908303*x^12 - 534454*x^10 + 5620755*x^8 + 1997133*x^6 + 2139445*x^4 - 21516*x^2 + 144}} \\[2mm]

17 & $979, 1$ & \texttt{\detokenize{x^16 - 82*x^14 + 873*x^12 + 54190*x^10 + 921916*x^8 + 275878*x^6 + 75153*x^4 - 1450*x^2 + 625}} \\[2mm]

18 & $1043, 1$ & \texttt{\detokenize{x^16 + 4*x^15 - 112*x^14 - 468*x^13 + 4524*x^12 + 26084*x^11 - 37200*x^10 - 620180*x^9 - 1129594*x^8 + 7177068*x^7 + 50733232*x^6 + 169055588*x^5 + 408075756*x^4 + 800699052*x^3 + 1206780816*x^2 + 1158039300*x + 496163425}} \\[2mm]

19 & $1803, 1$ & \texttt{\detokenize{x^16 + 17*x^14 + 73*x^12 + 5886*x^10 + 65502*x^8 - 238194*x^6 + 1231281*x^4 + 29889*x^2 + 729}} \\[2mm]

20 & $1939, 1$ & \texttt{\detokenize{x^16 + 83*x^14 + 434598*x^12 + 15022879*x^10 + 274274502*x^8 - 29397973*x^6 + 8747713*x^4 - 915768*x^2 + 63504}} \\[2mm]

21 & $2307, 1$ & \texttt{\detokenize{x^16 - 662*x^14 + 201819*x^12 - 11127950*x^10 + 279521689*x^8 + 506012376*x^6 + 626839920*x^4 - 13239936*x^2 + 186624}} \\[2mm]

22 & $2611, 1$ & \texttt{\detokenize{x^16 + 195*x^14 + 48282*x^12 - 1147485*x^10 + 49554874*x^8 + 142596759*x^6 + 127887069*x^4 + 27120324*x^2 + 1882384}} \\[2mm]

23 & $2947, 1$ & \texttt{\detokenize{x^16 + 1042*x^14 + 259533*x^12 - 20238982*x^10 + 1444870932*x^8 + 437441458*x^6 + 100165381*x^4 + 1195866*x^2 + 3969}} \\[2mm]

24 & $3787, 1$ & \texttt{\detokenize{x^16 + 4*x^15 - 1080*x^14 - 250*x^13 + 276931*x^12 - 958338*x^11 + 11629756*x^10 - 90980304*x^9 + 1013691008*x^8 - 4343238400*x^7 + 24885046080*x^6 - 142240543616*x^5 + 838149812480*x^4 - 2959997260800*x^3 + 6870224211968*x^2 - 8948585299968*x + 5588060995584}} \\[2mm]

25 & $3883, 1$ & \texttt{\detokenize{x^16 + 564*x^14 + 108882*x^12 - 1763688*x^10 + 7580467*x^8 + 27094008*x^6 + 120558186*x^4 + 3324828*x^2 + 1692601}} \\[2mm]

26 & $3963, 1$ & \texttt{\detokenize{x^16 + 523156*x^14 + 138793202940*x^12 + 303954552910*x^10 + 5724446388196*x^8 - 1633217194428*x^6 + 40428668068641*x^4 + 9408759768*x^2 + 944784}} \\[2mm]

27 & $5947, 1$ & \texttt{\detokenize{x^16 + 4*x^15 + 2018*x^14 + 12676*x^13 + 1057178*x^12 + 8518606*x^11 + 51899634*x^10 - 57129850*x^9 + 2654128326*x^8 + 27218837812*x^7 + 140809867094*x^6 - 615052835006*x^5 + 3621298696829*x^4 + 18116412464518*x^3 + 224354894493592*x^2 - 1610645712891096*x + 6275671421485392}} \\[2mm]

\end{longtable}

\begin{longtable}{|c|c|M|}
\caption{Imaginary non-CM \(D_{16}\) Fields with Class Number One}\label{tab:D16-ram}\\
\hline
\textbf{Entry} & $[d,\mathfrak{F}_{L/N}]$ & \textbf{Minimal Polynomial of \(L\)} \\
\hline
\endfirsthead
\hline
\textbf{Entry} & $[d,\mathfrak{F}_{L/N}]$ & \textbf{Minimal Polynomial of \(L\) } \\
\hline
\endhead
\hline
\endfoot

1 & $1, {\begin{smallmatrix}14&0\\0&14\end{smallmatrix}}$ & \texttt{\detokenize{x^16 + 4*x^15 - 6*x^14 - 28*x^13 + 63*x^12 + 56*x^11 - 238*x^10 + 284*x^9 + 373*x^8 - 1844*x^7 + 1316*x^6 + 1176*x^5 - 1876*x^4 + 2352*x^3 + 1296*x^2 - 864*x + 1296}} \\[2mm]

2 & $2, {\begin{smallmatrix}8&0\\0&8\end{smallmatrix}}$ & \texttt{\detokenize{x^16 + 6*x^8 + 1}} \\[2mm]

3 & $2, {\begin{smallmatrix}7&0\\0&7\end{smallmatrix}}$ & \texttt{\detokenize{x^16 + 4*x^15 - 44*x^14 - 196*x^13 + 434*x^12 + 2100*x^11 + 2688*x^10 + 284*x^9 + 177*x^8 - 5048*x^7 - 9828*x^6 - 6440*x^5 + 9660*x^4 + 12656*x^3 - 688*x^2 - 960*x + 1600}} \\[2mm]

4 & $3, {\begin{smallmatrix}32&0\\0&32\end{smallmatrix}}$ & \texttt{\detokenize{x^16 - 24*x^14 + 252*x^12 - 1088*x^10 + 2250*x^8 - 2304*x^6 + 1096*x^4 - 192*x^2 + 36}} \\[2mm]

5 & $3, {\begin{smallmatrix}32&0\\0&32\end{smallmatrix}}$ & \texttt{\detokenize{x^16 + 120*x^12 - 352*x^10 + 732*x^8 - 384*x^6 + 592*x^4 + 192*x^2 + 36}} \\[2mm]

6 & $11, {\begin{smallmatrix}7&0\\0&7\end{smallmatrix}}$ & \texttt{\detokenize{x^16 + 8*x^15 + 46*x^14 + 182*x^13 + 434*x^12 + 602*x^11 - 1022*x^10 - 7244*x^9 - 17667*x^8 - 26302*x^7 + 18858*x^6 + 117572*x^5 + 297276*x^4 + 382690*x^3 + 242708*x^2 + 73126*x + 14051}} \\[2mm]

7 & $11, {\begin{smallmatrix}32&0\\0&32\end{smallmatrix}}$ & \texttt{\detokenize{x^16 - 8*x^14 - 48*x^13 + 44*x^12 - 192*x^11 + 10488*x^10 + 50064*x^9 + 96336*x^8 + 555520*x^7 + 2810352*x^6 - 473888*x^5 - 14119384*x^4 - 3594880*x^3 + 34706192*x^2 - 19961952*x + 3136204}} \\[2mm]

8 & $19, {\begin{smallmatrix}32&0\\0&32\end{smallmatrix}}$ & \texttt{\detokenize{x^16 + 1136*x^14 + 354904*x^12 + 1037952*x^10 + 1133256*x^8 + 480000*x^6 - 49088*x^4 - 64768*x^2 + 18496}} \\[2mm]

9 & $19, {\begin{smallmatrix}32&0\\0&32\end{smallmatrix}}$ & \texttt{\detokenize{x^16 - 1136*x^14 + 354904*x^12 - 1037952*x^10 + 1133256*x^8 - 480000*x^6 - 49088*x^4 + 64768*x^2 + 18496}} \\[2mm]

10 & $43, {\begin{smallmatrix}7&0\\0&7\end{smallmatrix}}$ & \texttt{\detokenize{x^16 + 8*x^15 + 62*x^14 + 294*x^13 + 798*x^12 + 1330*x^11 + 2422*x^10 + 5972*x^9 + 107213*x^8 + 397354*x^7 + 1350398*x^6 + 2681700*x^5 + 3500700*x^4 + 2985402*x^3 + 800008*x^2 - 310458*x + 79479}} \\[2mm]

11 & $43, {\begin{smallmatrix}32&0\\0&32\end{smallmatrix}}$ & \texttt{\detokenize{x^16 - 21368*x^14 + 123925940*x^12 - 167615904*x^10 + 1046986946*x^8 - 770029440*x^6 + 490466360*x^4 - 126752*x^2 + 58564}} \\[2mm]

12 & $43, {\begin{smallmatrix}32&0\\0&32\end{smallmatrix}}$ & \texttt{\detokenize{x^16 + 21368*x^14 + 123925940*x^12 + 167615904*x^10 + 1046986946*x^8 + 770029440*x^6 + 490466360*x^4 + 126752*x^2 + 58564}} \\[2mm]

13 & $67, {\begin{smallmatrix}32&0\\0&32\end{smallmatrix}}$ & \texttt{\detokenize{x^16 + 424*x^14 - 20464*x^13 + 820700*x^12 - 4735040*x^11 + 296883640*x^10 - 7905385904*x^9 + 150762199952*x^8 - 28419707392*x^7 + 538808757776*x^6 + 77578572640*x^5 + 1494240312456*x^4 - 736845840512*x^3 + 2900913579280*x^2 - 1214728232160*x + 1606294408428}} \\[2mm]

14 & $67, {\begin{smallmatrix}32&0\\0&32\end{smallmatrix}}$ & \texttt{\detokenize{x^16 - 312*x^14 + 20464*x^13 + 807804*x^12 - 2770496*x^11 + 12061400*x^10 + 8033989040*x^9 + 153428795760*x^8 + 48767804672*x^7 + 921414508816*x^6 - 57376971104*x^5 + 1966133885512*x^4 + 470398857088*x^3 + 1273300771664*x^2 + 880532571360*x + 525898752300}} \\[2mm]

15 & $68, {\begin{smallmatrix}2&0\\0&2\end{smallmatrix}}$ & \texttt{\detokenize{x^16 + 10*x^14 + 85*x^12 + 370*x^10 + 1316*x^8 + 2770*x^6 + 1165*x^4 - 70*x^2 + 1}} \\[2mm]

16 & $292, {\begin{smallmatrix}2&0\\0&2\end{smallmatrix}}$ & \texttt{\detokenize{x^16 + 8*x^15 + 18*x^14 - 68*x^13 - 287*x^12 - 672*x^11 + 1994*x^10 + 12332*x^9 + 41800*x^8 + 86172*x^7 + 139882*x^6 + 166960*x^5 + 177213*x^4 + 159212*x^3 + 144746*x^2 + 92280*x + 42425}} \\[2mm]

17 & $388, {\begin{smallmatrix}2&0\\0&2\end{smallmatrix}}$ & \texttt{\detokenize{x^16 + 62*x^14 + 1803*x^12 + 31306*x^10 + 348001*x^8 + 2440720*x^6 + 10157856*x^4 + 22265600*x^2 + 19360000}} \\[2mm]

18 & $772, {\begin{smallmatrix}2&0\\0&2\end{smallmatrix}}$ & \texttt{\detokenize{x^16 + 8*x^15 + 30*x^14 - 64*x^13 - 631*x^12 - 1228*x^11 + 13678*x^10 + 114232*x^9 + 566624*x^8 + 1848516*x^7 + 4722642*x^6 + 8592892*x^5 + 12273065*x^4 + 11005632*x^3 + 7972434*x^2 + 1807596*x + 3734829}} \\[2mm]

\end{longtable}

\begin{longtable}{|c|c|M|}
\caption{Imaginary non-CM \(SD_{16}\) Fields with Class Number One}\label{tab:SD16}\\
\hline
\textbf{Entry} & $[d, \mathfrak{F}_{L/N}]$ & \textbf{Minimal Polynomial of \(L\)} \\
\hline
\endfirsthead
\hline
\textbf{Entry} & $[d, \mathfrak{F}_{L/N}]$ & \textbf{Minimal Polynomial of \(L\)} \\
\hline
\endhead
\hline
\endfoot

1 & $1, {\begin{smallmatrix}12&6\\0&6\end{smallmatrix}}$ & \texttt{\detokenize{x^16 - 6*x^12 + 39*x^8 + 18*x^4 + 9}} \\[2mm]

2 & $1, {\begin{smallmatrix}12&0\\0&12\end{smallmatrix}}$ & \texttt{\detokenize{x^16 + 6*x^12 + 39*x^8 - 18*x^4 + 9}} \\[2mm]

3 & $1, {\begin{smallmatrix}44&22\\0&22\end{smallmatrix}}$ & \texttt{\detokenize{x^16 + 8*x^15 + 32*x^14 + 172*x^13 + 590*x^12 + 1404*x^11 + 4504*x^10 + 6416*x^9 + 14463*x^8 + 12316*x^7 + 14216*x^6 - 6320*x^5 - 4810*x^4 + 11444*x^3 + 26560*x^2 + 53700*x + 46457}} \\[2mm]

4 & $1, {\begin{smallmatrix}44&0\\0&44\end{smallmatrix}}$ & \texttt{\detokenize{x^16 + 8*x^15 + 32*x^14 - 4*x^13 - 290*x^12 - 532*x^11 + 1952*x^10 + 4480*x^9 - 9561*x^8 - 20244*x^7 + 34104*x^6 + 49384*x^5 - 84186*x^4 - 33172*x^3 + 89216*x^2 - 45828*x + 23797}} \\[2mm]

5 & $1, {\begin{smallmatrix}76&38\\0&38\end{smallmatrix}}$ & \texttt{\detokenize{x^16 + 8*x^15 + 1248*x^14 + 8596*x^13 + 373758*x^12 + 2131164*x^11 - 2928456*x^10 - 33974032*x^9 + 185487231*x^8 + 968886428*x^7 + 2142328008*x^6 + 2870114736*x^5 + 9685520310*x^4 + 15797367700*x^3 + 25108770000*x^2 + 17515017500*x + 10615015625}} \\[2mm]

6 & $1, {\begin{smallmatrix}76&0\\0&76\end{smallmatrix}}$ & \texttt{\detokenize{x^16 + 381444*x^12 - 108737874*x^8 + 13910659924*x^4 + 29241}} \\[2mm]

7 & $1, {\begin{smallmatrix}172&86\\0&86\end{smallmatrix}}$ & \texttt{\detokenize{x^16 - 430*x^12 + 98943*x^8 - 2348230*x^4 + 154430329}} \\[2mm]

8 & $1, {\begin{smallmatrix}268&134\\0&134\end{smallmatrix}}$ & \texttt{\detokenize{x^16 + 1340*x^14 + 1064764*x^12 - 701113728*x^10 + 96396428220*x^8 + 131041882464*x^6 + 74791695856*x^4 + 20324755520*x^2 + 2149620496}} \\[2mm]

9 & $1, {\begin{smallmatrix}268&0\\0&268\end{smallmatrix}}$ & \texttt{\detokenize{x^16 + 2144*x^14 + 2131404*x^12 + 357179680*x^10 + 97101739900*x^8 - 66181506560*x^6 + 36551161776*x^4 - 8099017888*x^2 + 2149620496}} \\[2mm]

10 & $1, {\begin{smallmatrix}652&0\\0&652\end{smallmatrix}}$ & \texttt{\detokenize{x^16 - 15648*x^14 + 2339226366*x^12 - 2754884260416*x^10 + 832729823747127*x^8 - 8522788163419680*x^6 + 22187629036095790*x^4 + 16158821566320*x^2 + 6045218001}} \\[2mm]

11 & $2, {\begin{smallmatrix}8&0\\0&8\end{smallmatrix}}$ & \texttt{\detokenize{x^16 + 8*x^15 + 40*x^14 + 144*x^13 + 388*x^12 + 800*x^11 + 1256*x^10 + 1424*x^9 + 1028*x^8 + 320*x^7 + 160*x^6 + 192*x^5 + 272*x^4 - 192*x^3 + 8}} \\[2mm]

\end{longtable}

\begin{longtable}{|c|c|M|}
\caption{Imaginary non-CM \(M_4(2)\) Fields with Class Number One}\label{tab:M4(2)}\\
\hline
\textbf{Entry} & $[d, \mathfrak{F}_{L/N}]$ & \textbf{Minimal Polynomial of \(L\)} \\
\hline
\endfirsthead
\hline
\textbf{Entry} & $[d, \mathfrak{F}_{L/N}]$ & \textbf{Minimal Polynomial of \(L\)} \\
\hline
\endhead
\hline
\endfoot
1 & $2, {\begin{smallmatrix}16&0\\0&8\end{smallmatrix}}$ & \texttt{\detokenize{x^16 - 8*x^12 + 32*x^8 + 8*x^4 + 1}} \\[2mm]

2 & $-5, {\begin{smallmatrix}80&32\\0&16\end{smallmatrix}}$, [1,1] & \texttt{\detokenize{x^16 + 4*x^15 - 30*x^14 - 60*x^13 + 155*x^12 - 708*x^11 + 2948*x^10 + 3760*x^9 + 7835*x^8 + 21120*x^7 + 16508*x^6 - 8148*x^5 + 12245*x^4 - 5700*x^3 + 5250*x^2 - 2476*x + 571}} \\[2mm]

3 & $-5, {\begin{smallmatrix}80&32\\0&16\end{smallmatrix}}$, [1,1] & \texttt{\detokenize{x^16 - 8*x^15 + 44*x^14 - 128*x^13 + 418*x^12 - 2168*x^11 + 12696*x^10 - 49064*x^9 + 134703*x^8 - 326408*x^7 + 1084224*x^6 - 4133064*x^5 + 12506798*x^4 - 26430576*x^3 + 37219276*x^2 - 31881784*x + 12855161}} \\[2mm]

4 & $-5, {\begin{smallmatrix}15&6\\0&3\end{smallmatrix}}$, [1,1] & \texttt{\detokenize{x^16 - 2*x^15 - 10*x^14 + 30*x^13 + 30*x^12 - 166*x^11 + 67*x^10 + 280*x^9 - 500*x^8 + 600*x^7 + 702*x^6 - 2624*x^5 + 405*x^4 + 1155*x^3 + 1630*x^2 - 398*x + 271}} \\[2mm]

5 & $-5, {\begin{smallmatrix}35&14\\0&7\end{smallmatrix}}$, [1,1] & \texttt{\detokenize{x^16 - 4*x^15 - 15*x^14 - 30*x^13 + 230*x^12 + 918*x^11 + 1898*x^10 + 1455*x^9 + 720*x^8 + 2315*x^7 + 12843*x^6 + 28108*x^5 + 40770*x^4 + 37275*x^3 + 24260*x^2 + 8976*x + 1331}} \\[2mm]

\end{longtable}

\begin{longtable}{|c|c|M|}
\caption{Imaginary non-CM \(C_4\circ D_8\) Fields with Class Number One}\label{tab:C4OD8-ram}\\
\hline
\textbf{Entry} & $[d, \mathfrak{F}_{L/N}]$ & \textbf{Minimal Polynomial of \(L\)} \\
\hline
\endfirsthead
\hline
\textbf{Entry} & $[d, \mathfrak{F}_{L/N}]$ & \textbf{Minimal Polynomial of \(L\)} \\
\hline
\endhead
\hline
\endfoot

1 & $1, {\begin{smallmatrix}88&44\\0&44\end{smallmatrix}}$ & \texttt{\detokenize{x^16 + 4*x^12 + 880*x^10 + 974*x^8 - 3872*x^6 + 71636*x^4 - 254672*x^2 + 235225}} \\[2mm]

2 & $2, {\begin{smallmatrix}12&0\\0&6\end{smallmatrix}}$ & \texttt{\detokenize{x^16 + 8*x^15 + 36*x^14 + 96*x^13 + 230*x^12 + 640*x^11 + 1312*x^10 + 2604*x^9 + 6013*x^8 + 9520*x^7 + 13708*x^6 + 25508*x^5 + 39768*x^4 + 52080*x^3 + 65616*x^2 + 57420*x + 21601}} \\[2mm]

3 & $2, {\begin{smallmatrix}24&0\\0&12\end{smallmatrix}}$ & \texttt{\detokenize{x^16 + 8*x^15 + 20*x^14 + 32*x^13 + 198*x^12 + 792*x^11 + 1500*x^10 + 2264*x^9 + 5977*x^8 + 20736*x^7 + 54672*x^6 + 90648*x^5 + 101808*x^4 + 77544*x^3 + 40608*x^2 + 12960*x + 2025}} \\[2mm]

4 & $1, {\begin{smallmatrix}24&12\\0&12\end{smallmatrix}}$ & \texttt{\detokenize{x^16 - 12*x^14 + 40*x^12 + 36*x^10 - 162*x^8 - 1140*x^6 + 3016*x^4 - 1956*x^2 + 9409}} \\[2mm]

5 & $1, {\begin{smallmatrix}15&0\\0&15\end{smallmatrix}}$ & \texttt{\detokenize{x^16 - 4*x^15 + 18*x^14 - 38*x^13 + 90*x^12 - 150*x^11 + 292*x^10 - 494*x^9 + 600*x^8 - 550*x^7 + 244*x^6 - 516*x^5 + 1953*x^4 - 2872*x^3 + 2076*x^2 - 770*x + 121}} \\[2mm]

6 & $2, {\begin{smallmatrix}44&0\\0&22\end{smallmatrix}}$ & \texttt{\detokenize{x^16 + 8*x^15 + 52*x^14 + 216*x^13 + 862*x^12 + 3072*x^11 + 9080*x^10 + 13884*x^9 - 24707*x^8 - 204272*x^7 - 535172*x^6 - 561780*x^5 + 664048*x^4 + 3228968*x^3 + 5077712*x^2 + 4069180*x + 1426825}} \\[2mm]

7 & $2, {\begin{smallmatrix}88&0\\0&44\end{smallmatrix}}$ & \texttt{\detokenize{x^16 + 8*x^15 + 60*x^14 + 232*x^13 + 870*x^12 + 2424*x^11 + 7140*x^10 + 16064*x^9 + 16537*x^8 + 32016*x^7 + 53040*x^6 - 13760*x^5 - 58952*x^4 + 183496*x^3 + 126808*x^2 - 333168*x + 1380017}} \\[2mm]

8 & $3, {\begin{smallmatrix}24&8\\0&8\end{smallmatrix}}$ & \texttt{\detokenize{x^16 - 4*x^14 + 16*x^12 - 148*x^10 + 502*x^8 - 28*x^6 + 18976*x^4 + 27956*x^2 + 11881}} \\[2mm]

9 & $3, {\begin{smallmatrix}48&16\\0&16\end{smallmatrix}}$ & \texttt{\detokenize{x^16 - 24*x^14 - 96*x^13 + 196*x^12 + 2112*x^11 + 4512*x^10 - 7584*x^9 - 64980*x^8 - 116736*x^7 + 171072*x^6 + 1263552*x^5 + 2898064*x^4 + 3872256*x^3 + 3442752*x^2 + 1859712*x + 446224}} \\[2mm]

10 & $3, {\begin{smallmatrix}48&16\\0&16\end{smallmatrix}}$ & \texttt{\detokenize{x^16 + 8*x^15 + 4*x^14 - 160*x^13 - 410*x^12 + 776*x^11 + 4168*x^10 + 1208*x^9 - 10565*x^8 - 13896*x^7 + 4776*x^6 + 13800*x^5 - 522*x^4 - 16560*x^3 + 51588*x^2 + 26712*x + 14481}} \\[2mm]

11 & $-2, {\begin{smallmatrix}44&0\\0&22\end{smallmatrix}}$, [1,1] & \texttt{\detokenize{x^16 - 8*x^15 - 32*x^14 + 232*x^13 + 528*x^12 - 2192*x^11 + 480*x^10 + 36320*x^9 + 49216*x^8 - 94464*x^7 - 117760*x^6 + 685632*x^5 + 2113408*x^4 + 2683264*x^3 + 2017024*x^2 + 1138688*x + 405056}} \\[2mm]

12 & $-2, {\begin{smallmatrix}88&0\\0&44\end{smallmatrix}}$, [1,1] & \texttt{\detokenize{x^16 - 88*x^13 - 436*x^12 + 176*x^11 + 3872*x^10 + 20240*x^9 + 46800*x^8 - 19008*x^7 - 77440*x^6 - 159632*x^5 + 229816*x^4 + 594528*x^3 + 557568*x^2 + 361152*x + 116964}} \\[2mm]

13 & $-5, {\begin{smallmatrix}16&0\\0&16\end{smallmatrix}}$, [1,1] & \texttt{\detokenize{x^16 + 16*x^14 + 228*x^12 + 2288*x^10 + 14460*x^8 + 50656*x^6 + 70960*x^4 + 22464*x^2 + 1296}} \\[2mm]

14 & $11, {\begin{smallmatrix}176&96\\0&16\end{smallmatrix}}$ & \texttt{\detokenize{x^16 + 8*x^15 + 68*x^14 + 160*x^13 + 642*x^12 - 328*x^11 + 6168*x^10 + 216*x^9 + 60951*x^8 - 193048*x^7 + 351936*x^6 - 501064*x^5 + 525990*x^4 - 1084880*x^3 + 6825796*x^2 - 6382264*x + 4692769}} \\[2mm]

15 & $15, {\begin{smallmatrix}8&0\\0&8\end{smallmatrix}}$ & \texttt{\detokenize{x^16 - 6*x^15 + 7*x^14 + 56*x^13 - 141*x^12 - 100*x^11 + 725*x^10 + 224*x^9 - 714*x^8 - 1404*x^7 + 11635*x^6 - 16648*x^5 - 6783*x^4 + 48818*x^3 + 76906*x^2 + 45176*x + 11929}} \\[2mm]

16 & $10, {\begin{smallmatrix}8&0\\0&4\end{smallmatrix}}$ & \texttt{\detokenize{x^16 + 8*x^14 + 120*x^13 + 260*x^12 + 960*x^11 + 5592*x^10 + 17280*x^9 + 47780*x^8 + 123520*x^7 + 269040*x^6 + 534720*x^5 + 858528*x^4 + 1183680*x^3 + 1435936*x^2 + 1311680*x + 952976}} \\[2mm]

\end{longtable}

\begin{longtable}{|c|c|M|}
\caption{Imaginary non-CM $C_2 \times D_8$ fields with class number one }\label{tab:C2D8-ram}\\
\hline
\textbf{Entry} & \([d,\mathfrak{F}_{L/N}]\) & \textbf{Minimal Polynomial of \(L\) } \\
\hline
\endfirsthead
\hline
\textbf{Entry} & \([d,\mathfrak{F}_{L/N}]\) & \textbf{Minimal Polynomial of \(L\)} \\
\hline
\endhead
\hline
\endfoot

1 & $7, {\begin{smallmatrix}12&0\\0&12\end{smallmatrix}}$ & \texttt{\detokenize{x^16 - 6*x^15 + 5*x^14 + 26*x^13 - 34*x^12 - 24*x^11 - 55*x^10 + 222*x^9 - 101*x^8 - 330*x^7 + 1159*x^6 - 1676*x^5 + 1686*x^4 - 2042*x^3 + 1609*x^2 - 2852*x + 3637}} \\[2mm]

2 & $7, {\begin{smallmatrix}24&0\\0&24\end{smallmatrix}}$ & \texttt{\detokenize{x^16 + 4*x^15 + 10*x^14 + 28*x^13 + 71*x^12 + 80*x^11 - 112*x^10 - 212*x^9 + 694*x^8 + 2300*x^7 + 3030*x^6 + 2300*x^5 + 3067*x^4 - 2576*x^3 + 424*x^2 + 4604*x + 1759}} \\[2mm]

3 & $7, {\begin{smallmatrix}24&0\\0&24\end{smallmatrix}}$ & \texttt{\detokenize{x^16 + 2*x^15 + 25*x^14 + 26*x^13 + 226*x^12 + 188*x^11 + 955*x^10 + 908*x^9 + 2101*x^8 + 2076*x^7 + 2943*x^6 + 1956*x^5 + 4710*x^4 + 1278*x^3 + 4005*x^2 - 162*x + 603}} \\[2mm]

4 & $7, {\begin{smallmatrix}76&0\\0&76\end{smallmatrix}}$ & \texttt{\detokenize{x^16 - 6*x^15 + x^14 + 100*x^13 + 201*x^12 - 562*x^11 + 470*x^10 + 9090*x^9 + 12461*x^8 + 5392*x^7 + 15322*x^6 - 38166*x^5 - 4192*x^4 - 163342*x^3 + 592248*x^2 - 606326*x + 759257}} \\[2mm]

5 & $7, {\begin{smallmatrix}15&0\\0&15\end{smallmatrix}}$ & \texttt{\detokenize{x^16 - 2*x^15 + 17*x^14 - 44*x^13 + 196*x^12 - 406*x^11 + 1157*x^10 - 1696*x^9 + 3541*x^8 - 3372*x^7 + 5709*x^6 - 2742*x^5 + 5088*x^4 + 612*x^3 + 1881*x^2 + 2700*x + 711}} \\[2mm]

6 & $-13, {\begin{smallmatrix}28&0\\0&28\end{smallmatrix}}$, [1,1] & \texttt{\detokenize{x^16 - 8*x^15 + 84*x^14 - 448*x^13 + 2446*x^12 - 9216*x^11 + 32420*x^10 - 87344*x^9 + 207697*x^8 - 391512*x^7 + 633136*x^6 - 811984*x^5 + 878904*x^4 - 728480*x^3 + 509504*x^2 - 235200*x + 84800}} \\[2mm]

7 & $1, {\begin{smallmatrix}35&0\\0&35\end{smallmatrix}}$ & \texttt{\detokenize{x^16 - 4*x^15 + 18*x^14 - 42*x^13 + 102*x^12 - 166*x^11 + 300*x^10 - 562*x^9 + 657*x^8 - 870*x^7 + 826*x^6 - 988*x^5 + 1624*x^4 - 1416*x^3 + 1384*x^2 - 432*x + 144}} \\[2mm]

8 & $1, {\begin{smallmatrix}70&0\\0&70\end{smallmatrix}}$ & \texttt{\detokenize{x^16 - 4*x^15 - 2*x^14 + 34*x^13 - 12*x^12 - 110*x^11 + 110*x^10 - 162*x^9 + 457*x^8 + 410*x^7 - 576*x^6 - 2088*x^5 + 1006*x^4 + 1960*x^3 - 80*x^2 + 400*x + 2900}} \\[2mm]

9 & $1, {\begin{smallmatrix}44&0\\0&44\end{smallmatrix}}$ & \texttt{\detokenize{x^16 - 8*x^15 + 32*x^14 - 92*x^13 + 206*x^12 - 364*x^11 + 552*x^10 - 1048*x^9 + 2623*x^8 - 5180*x^7 + 6456*x^6 - 3896*x^5 + 222*x^4 + 972*x^3 + 392*x^2 + 84*x + 9}} \\[2mm]

10 & $1, {\begin{smallmatrix}88&44\\0&44\end{smallmatrix}}$ & \texttt{\detokenize{x^16 - 8*x^15 + 24*x^14 - 20*x^13 + 10*x^12 - 180*x^11 - 36*x^10 + 1688*x^9 + 1309*x^8 - 11164*x^7 - 5432*x^6 + 35112*x^5 + 29220*x^4 - 67188*x^3 - 81708*x^2 + 48076*x + 129017}} \\[2mm]

11 & $3, {\begin{smallmatrix}187&0\\0&187\end{smallmatrix}}$ & \texttt{\detokenize{x^16 - 4*x^15 + 584*x^14 - 2120*x^13 + 147940*x^12 - 441780*x^11 + 14253710*x^10 - 39202136*x^9 + 15505132*x^8 - 1374887416*x^7 - 95028942412*x^6 + 123155298084*x^5 + 5097833532049*x^4 - 1994249390620*x^3 - 103978917426100*x^2 - 5655161827376*x + 814053076246864}} \\[2mm]

12 & $2, {\begin{smallmatrix}20&0\\0&20\end{smallmatrix}}$ & \texttt{\detokenize{x^16 - 8*x^15 + 80*x^14 - 416*x^13 + 2572*x^12 - 10352*x^11 + 46280*x^10 - 149664*x^9 + 534968*x^8 - 1403984*x^7 + 3966000*x^6 - 8121792*x^5 + 19695752*x^4 - 28782976*x^3 + 53983440*x^2 - 31740288*x + 115981156}} \\[2mm]

13 & $2, {\begin{smallmatrix}40&0\\0&20\end{smallmatrix}}$ & \texttt{\detokenize{x^16 + 16*x^14 - 8*x^13 + 144*x^12 - 32*x^11 + 616*x^10 + 552*x^9 + 1450*x^8 + 3648*x^7 + 1776*x^6 + 9528*x^5 + 12656*x^4 + 8928*x^3 + 8264*x^2 - 6936*x + 17001}} \\[2mm]

14 & $3, {\begin{smallmatrix}28&0\\0&28\end{smallmatrix}}$ & \texttt{\detokenize{x^16 - 26*x^14 + 275*x^12 - 1458*x^10 + 4613*x^8 - 8880*x^6 + 8828*x^4 - 320*x^2 + 16}} \\[2mm]

15 & $1, {\begin{smallmatrix}20&0\\0&20\end{smallmatrix}}$ & \texttt{\detokenize{x^16 - 8*x^15 + 40*x^14 - 136*x^13 + 368*x^12 - 836*x^11 + 1618*x^10 - 2480*x^9 + 2447*x^8 - 332*x^7 - 3514*x^6 + 6100*x^5 - 4134*x^4 - 960*x^3 + 3990*x^2 - 3100*x + 961}} \\[2mm]

16 & $1, {\begin{smallmatrix}40&20\\0&20\end{smallmatrix}}$ & \texttt{\detokenize{x^16 - 8*x^15 + 56*x^14 - 252*x^13 + 950*x^12 - 2788*x^11 + 6824*x^10 - 13616*x^9 + 22641*x^8 - 30956*x^7 + 34900*x^6 - 31864*x^5 + 23368*x^4 - 13308*x^3 + 5704*x^2 - 1652*x + 281}} \\[2mm]

17 & $1, {\begin{smallmatrix}21&0\\0&21\end{smallmatrix}}$ & \texttt{\detokenize{x^16 - 4*x^15 + 6*x^14 - 14*x^13 + 12*x^12 + 6*x^11 - 20*x^10 + 160*x^9 + 1329*x^8 - 5698*x^7 + 9772*x^6 - 10104*x^5 + 18048*x^4 - 19966*x^3 + 11478*x^2 - 12050*x + 11701}} \\[2mm]

18 & $2, {\begin{smallmatrix}12&0\\0&12\end{smallmatrix}}$ & \texttt{\detokenize{x^16 - 12*x^14 - 8*x^13 + 66*x^12 + 32*x^11 - 124*x^10 - 152*x^9 + 255*x^8 + 224*x^7 - 108*x^6 - 184*x^5 + 26*x^4 - 272*x^3 + 392*x^2 - 176*x + 121}} \\[2mm]

19 & $2, {\begin{smallmatrix}12&0\\0&12\end{smallmatrix}}$ & \texttt{\detokenize{x^16 + 16*x^14 + 52*x^13 + 138*x^12 + 460*x^11 + 1220*x^10 + 2464*x^9 + 4317*x^8 + 8004*x^7 + 16624*x^6 + 25312*x^5 + 34164*x^4 + 85348*x^3 + 172516*x^2 + 176812*x + 76777}} \\[2mm]

20 & $2, {\begin{smallmatrix}24&0\\0&12\end{smallmatrix}}$ & \texttt{\detokenize{x^16 + 8*x^13 + 16*x^11 + 32*x^10 + 8*x^9 + 270*x^8 + 352*x^7 + 192*x^6 + 376*x^5 + 464*x^4 + 80*x^3 + 32*x^2 + 56*x + 49}} \\[2mm]

21 & $2, {\begin{smallmatrix}24&0\\0&12\end{smallmatrix}}$ & \texttt{\detokenize{x^16 - 8*x^14 - 8*x^13 + 40*x^12 + 80*x^11 - 88*x^10 - 104*x^9 - 330*x^8 - 32*x^7 + 2344*x^6 - 2872*x^5 - 440*x^4 + 4112*x^3 + 696*x^2 - 1432*x + 1129}} \\[2mm]

22 & $2, {\begin{smallmatrix}21&0\\0&21\end{smallmatrix}}$ & \texttt{\detokenize{x^16 + 4*x^15 + 36*x^13 + 250*x^12 + 512*x^11 + 1812*x^10 + 7500*x^9 + 16971*x^8 + 24240*x^7 + 41100*x^6 + 83084*x^5 + 112906*x^4 + 84732*x^3 + 41748*x^2 + 44920*x + 40009}} \\[2mm]

23 & $3, {\begin{smallmatrix}76&0\\0&76\end{smallmatrix}}$ & \texttt{\detokenize{x^16 - 8*x^15 + 156*x^13 - 238*x^12 - 1100*x^11 + 1884*x^10 + 3296*x^9 + 4251*x^8 - 15596*x^7 - 72948*x^6 + 148904*x^5 + 158930*x^4 - 368700*x^3 + 111204*x^2 - 278548*x + 360037}} \\[2mm]

24 & $1, {\begin{smallmatrix}133&0\\0&133\end{smallmatrix}}$ & \texttt{\detokenize{x^16 - 4*x^15 + 278*x^14 - 1078*x^13 + 33992*x^12 - 106778*x^11 + 2245824*x^10 - 4757456*x^9 + 84471697*x^8 - 61562290*x^7 + 1815736272*x^6 + 488471480*x^5 + 34581374688*x^4 + 8548605282*x^3 + 266628693434*x^2 + 496414211574*x + 1807103173257}} \\[2mm]

25 & $1, {\begin{smallmatrix}12&0\\0&12\end{smallmatrix}}$ & \texttt{\detokenize{x^16 - 4*x^15 + 8*x^14 - 16*x^13 + 70*x^12 - 188*x^11 + 320*x^10 - 524*x^9 + 799*x^8 - 876*x^7 + 768*x^6 - 612*x^5 + 354*x^4 - 144*x^3 + 72*x^2 - 36*x + 9}} \\[2mm]

26 & $1, {\begin{smallmatrix}91&0\\0&91\end{smallmatrix}}$ & \texttt{\detokenize{x^16 - 4*x^15 + 34*x^14 - 98*x^13 + 450*x^12 - 1014*x^11 + 3240*x^10 - 6370*x^9 + 14421*x^8 - 24246*x^7 + 41646*x^6 - 56180*x^5 + 82696*x^4 - 73672*x^3 + 83800*x^2 - 25680*x + 7632}} \\[2mm]

27 & $1, {\begin{smallmatrix}30&0\\0&30\end{smallmatrix}}$ & \texttt{\detokenize{x^16 - 4*x^15 + 14*x^14 - 38*x^13 + 96*x^12 - 166*x^11 + 262*x^10 - 278*x^9 + 109*x^8 + 394*x^7 - 776*x^6 + 868*x^5 + 974*x^4 - 720*x^3 + 160*x^2 + 400*x + 100}} \\[2mm]

28 & $2, {\begin{smallmatrix}35&0\\0&35\end{smallmatrix}}$ & \texttt{\detokenize{x^16 - 4*x^15 - 10*x^14 + 28*x^13 + 331*x^12 - 664*x^11 - 2748*x^10 + 8852*x^9 + 538*x^8 + 1828*x^7 - 63938*x^6 + 14164*x^5 + 1126071*x^4 - 960648*x^3 - 1032360*x^2 + 607644*x + 5121971}} \\[2mm]

29 & $3, {\begin{smallmatrix}44&0\\0&44\end{smallmatrix}}$ & \texttt{\detokenize{x^16 - 6*x^15 + 5*x^14 + 36*x^13 - 100*x^12 + 132*x^11 - 129*x^10 - 468*x^9 + 2912*x^8 - 8316*x^7 + 20031*x^6 - 43176*x^5 + 73391*x^4 - 92898*x^3 + 82730*x^2 - 43980*x + 10708}} \\[2mm]

30 & $2, {\begin{smallmatrix}44&0\\0&44\end{smallmatrix}}$ & \texttt{\detokenize{x^16 - 20*x^14 - 24*x^13 + 186*x^12 + 224*x^11 - 540*x^10 - 1512*x^9 + 1727*x^8 + 1312*x^7 + 676*x^6 - 3336*x^5 + 6138*x^4 - 11920*x^3 + 5952*x^2 + 1040*x + 977}} \\[2mm]

31 & $2, {\begin{smallmatrix}88&0\\0&44\end{smallmatrix}}$ & \texttt{\detokenize{x^16 - 8*x^15 + 52*x^14 - 240*x^13 + 942*x^12 - 2912*x^11 + 7652*x^10 - 15008*x^9 + 21279*x^8 - 11760*x^7 - 17220*x^6 + 63424*x^5 - 57890*x^4 + 21384*x^3 + 59600*x^2 - 45792*x + 9153}} \\[2mm]

32 & $2, {\begin{smallmatrix}88&0\\0&44\end{smallmatrix}}$ & \texttt{\detokenize{x^16 + 24*x^14 + 40*x^13 + 328*x^12 + 624*x^11 + 3144*x^10 + 4360*x^9 + 13590*x^8 + 12448*x^7 + 29064*x^6 + 24600*x^5 + 22568*x^4 + 4528*x^3 - 34152*x^2 - 7944*x + 11273}} \\[2mm]

33 & $-33, {\begin{smallmatrix}8&0\\0&8\end{smallmatrix}}$, [1,1] & \texttt{\detokenize{x^16 - 4*x^15 - 16*x^14 + 24*x^13 + 258*x^12 + 304*x^11 - 1144*x^10 - 3608*x^9 + 41*x^8 + 18244*x^7 + 46920*x^6 + 78912*x^5 + 127780*x^4 + 193296*x^3 + 211824*x^2 + 136512*x + 38592}} \\[2mm]

34 & $-2, {\begin{smallmatrix}12&0\\0&12\end{smallmatrix}}$, [1,1] & \texttt{\detokenize{x^16 - 8*x^15 + 24*x^14 - 16*x^13 - 36*x^12 + 184*x^10 - 688*x^9 + 1296*x^8 + 3952*x^7 - 14432*x^6 + 19968*x^5 + 11928*x^4 - 48416*x^3 + 43248*x^2 - 40480*x + 24100}} \\[2mm]

35 & $11, {\begin{smallmatrix}57&0\\0&57\end{smallmatrix}}$ & \texttt{\detokenize{x^16 + 25*x^14 + 2*x^13 + 622*x^12 + 1092*x^11 + 11308*x^10 + 36138*x^9 - 5114*x^8 + 784966*x^7 + 1280446*x^6 + 3948956*x^5 + 17632955*x^4 + 9240014*x^3 + 52564033*x^2 + 18287840*x + 42724300}} \\[2mm]

36 & $11, {\begin{smallmatrix}12&0\\0&12\end{smallmatrix}}$ & \texttt{\detokenize{x^16 - 22*x^14 + 241*x^12 - 1690*x^10 + 7732*x^8 - 12190*x^6 + 16561*x^4 - 778*x^2 + 961}} \\[2mm]

37 & $19, {\begin{smallmatrix}33&0\\0&33\end{smallmatrix}}$ & \texttt{\detokenize{x^16 - 6*x^15 - 36*x^14 + 210*x^13 + 1005*x^12 - 5760*x^11 - 11601*x^10 + 87804*x^9 + 97038*x^8 - 993978*x^7 - 37071*x^6 + 7368570*x^5 - 2016360*x^4 - 38211912*x^3 + 31159728*x^2 + 55142208*x + 81513216}} \\[2mm]

38 & $19, {\begin{smallmatrix}28&0\\0&28\end{smallmatrix}}$ & \texttt{\detokenize{x^16 - 4*x^15 + 26*x^14 - 74*x^13 + 268*x^12 - 682*x^11 + 1782*x^10 - 2310*x^9 - 4495*x^8 + 33006*x^7 - 82876*x^6 + 64896*x^5 + 130198*x^4 - 350064*x^3 + 390472*x^2 - 265472*x + 86228}} \\[2mm]

39 & $19, {\begin{smallmatrix}12&0\\0&12\end{smallmatrix}}$ & \texttt{\detokenize{x^16 - 6*x^15 - 25*x^14 + 176*x^13 + 263*x^12 - 2082*x^11 - 1294*x^10 + 13092*x^9 - 2840*x^8 - 37464*x^7 + 69016*x^6 - 59024*x^5 - 226656*x^4 + 626656*x^3 + 58336*x^2 - 1459904*x + 1243456}} \\[2mm]

40 & $-57, {\begin{smallmatrix}8&0\\0&8\end{smallmatrix}}$, [1,1] & \texttt{\detokenize{x^16 + 12*x^15 + 142*x^14 + 876*x^13 + 5553*x^12 + 23520*x^11 + 103588*x^10 + 308340*x^9 + 805699*x^8 + 1312368*x^7 + 2177448*x^6 - 6799392*x^5 - 9477360*x^4 + 26186112*x^3 + 115793280*x^2 - 212792832*x + 165618432}} \\[2mm]

41 & $-11, {\begin{smallmatrix}3&0\\0&3\end{smallmatrix}}$, [1,1] & \texttt{\detokenize{x^16 - 20*x^14 + 168*x^12 - 734*x^10 + 1720*x^8 - 1944*x^6 + 753*x^4 + 540*x^2 + 144}} \\[2mm]

42 & $187, {\begin{smallmatrix}3&0\\0&3\end{smallmatrix}}$ & \texttt{\detokenize{x^16 - 4*x^15 + 458*x^14 - 1662*x^13 + 80594*x^12 - 244214*x^11 + 7086196*x^10 - 15936738*x^9 + 345173482*x^8 - 508464810*x^7 + 9794514790*x^6 - 7744837810*x^5 + 162394311395*x^4 - 41087085234*x^3 + 1469423072756*x^2 + 86244327520*x + 5739720792832}} \\[2mm]

43 & $15, {\begin{smallmatrix}7&0\\0&7\end{smallmatrix}}$ & \texttt{\detokenize{x^16 - 8*x^15 + 94*x^14 - 518*x^13 + 3584*x^12 - 15134*x^11 + 71274*x^10 - 229034*x^9 + 769091*x^8 - 1849186*x^7 + 5157390*x^6 - 9815890*x^5 + 26507033*x^4 - 38426500*x^3 + 60449707*x^2 - 42621904*x + 32947861}} \\[2mm]

44 & $10, {\begin{smallmatrix}3&0\\0&3\end{smallmatrix}}$ & \texttt{\detokenize{x^16 + 8*x^15 + 84*x^14 + 448*x^13 + 2588*x^12 + 10068*x^11 + 38914*x^10 + 112004*x^9 + 310599*x^8 + 664532*x^7 + 1323298*x^6 + 2021100*x^5 + 2669438*x^4 + 2568244*x^3 + 1741926*x^2 + 710444*x + 160579}} \\[2mm]

45 & $22, {\begin{smallmatrix}3&0\\0&3\end{smallmatrix}}$ & \texttt{\detokenize{x^16 + 168*x^14 + 108*x^13 + 11652*x^12 + 14472*x^11 + 436230*x^10 + 768420*x^9 + 9659979*x^8 + 20547432*x^7 + 129951270*x^6 + 289379628*x^5 + 1026696114*x^4 + 2006608464*x^3 + 4004772678*x^2 + 5049493344*x + 3155036751}} \\[2mm]
\end{longtable}

\begin{longtable}{|c|c|M|}
\caption{Imaginary non-CM $C_2^2 \rtimes C_4$ fields with class number one}\label{tab:C22C4}\\
\hline
\textbf{Entry} & \([d,\mathfrak{F}_{L/N}]\) & \textbf{Minimal Polynomial of \(L\)} \\
\hline
\endfirsthead
\hline
\textbf{Entry} & \([d,\mathfrak{F}_{L/N}]\) & \textbf{Minimal Polynomial of \(L\) } \\
\hline
\endhead
\hline
\endfoot

1 & $7, {\begin{smallmatrix}37&0\\0&37\end{smallmatrix}}$ & \texttt{\detokenize{x^16 + 6*x^15 + 28*x^14 - 236*x^13 + 43*x^12 - 798*x^11 + 23310*x^10 - 60137*x^9 - 56976*x^8 + 372094*x^7 - 745617*x^6 + 659202*x^5 + 1184088*x^4 - 2964449*x^3 + 5200918*x^2 - 4456197*x + 2831229}} \\[2mm]

2 & $7, {\begin{smallmatrix}16&0\\0&16\end{smallmatrix}}$ & \texttt{\detokenize{x^16 + 4*x^15 - 2*x^14 + 93*x^12 + 160*x^11 + 278*x^10 + 1268*x^9 + 2456*x^8 + 2316*x^7 + 2398*x^6 + 5232*x^5 + 13581*x^4 + 24304*x^3 + 22046*x^2 + 5980*x + 529}} \\[2mm]

3 & $7, {\begin{smallmatrix}29&0\\0&29\end{smallmatrix}}$ & \texttt{\detokenize{x^16 + 48*x^14 - 42*x^13 + 1224*x^12 - 1512*x^11 + 14690*x^10 - 1470*x^9 + 92977*x^8 - 8022*x^7 - 119066*x^6 - 98952*x^5 + 1113601*x^4 - 149604*x^3 + 2067337*x^2 - 17780*x + 967289}} \\[2mm]

4 & $1, {\begin{smallmatrix}8&0\\0&8\end{smallmatrix}}$ & \texttt{\detokenize{x^16 - 16*x^14 + 80*x^12 - 96*x^10 - 158*x^8 + 48*x^6 + 528*x^4 + 352*x^2 + 289}} \\[2mm]

5 & $3, {\begin{smallmatrix}13&0\\0&13\end{smallmatrix}}$ & \texttt{\detokenize{x^16 - 16*x^14 + 6*x^13 + 98*x^12 - 72*x^11 - 156*x^10 + 492*x^9 + 44*x^8 - 1086*x^7 + 1266*x^6 - 354*x^5 - 2167*x^4 + 5400*x^3 + 7010*x^2 + 1824*x + 139}} \\[2mm]

6 & $1, {\begin{smallmatrix}26&0\\0&26\end{smallmatrix}}$ & \texttt{\detokenize{x^16 - 8*x^14 + 50*x^12 - 84*x^10 - 829*x^8 + 1716*x^6 + 7202*x^4 + 676*x^2 + 169}} \\[2mm]

7 & $1, {\begin{smallmatrix}52&0\\0&52\end{smallmatrix}}$ & \texttt{\detokenize{x^16 - 10*x^14 + 60*x^13 + 211*x^12 + 96*x^11 + 1738*x^10 - 720*x^9 - 9208*x^8 + 36744*x^7 + 73174*x^6 - 16932*x^5 + 124074*x^4 + 283824*x^3 + 297756*x^2 + 564732*x + 613089}} \\[2mm]

8 & $1, {\begin{smallmatrix}20&0\\0&20\end{smallmatrix}}$ & \texttt{\detokenize{x^16 - 18*x^14 - 4*x^13 + 139*x^12 + 64*x^11 - 566*x^10 - 328*x^9 + 1152*x^8 + 888*x^7 - 1146*x^6 - 1612*x^5 + 2466*x^4 - 1100*x^2 + 420*x + 305}} \\[2mm]

9 & $1, {\begin{smallmatrix}10&0\\0&10\end{smallmatrix}}$ & \texttt{\detokenize{x^16 + 4*x^15 + 4*x^14 + 40*x^13 + 238*x^12 + 628*x^11 + 1508*x^10 + 3512*x^9 + 4567*x^8 + 1524*x^7 - 2308*x^6 - 2880*x^5 - 1670*x^4 + 416*x^3 + 960*x^2 - 304*x + 181}} \\[2mm]

10 & $1, {\begin{smallmatrix}74&0\\0&74\end{smallmatrix}}$ & \texttt{\detokenize{x^16 - 24*x^14 + 290*x^12 - 1308*x^10 + 371*x^8 + 25308*x^6 + 103378*x^4 + 147852*x^2 + 110889}} \\[2mm]

11 & $1, {\begin{smallmatrix}148&0\\0&148\end{smallmatrix}}$ & \texttt{\detokenize{x^16 - 66*x^14 - 4*x^13 + 1851*x^12 + 272*x^11 - 21190*x^10 - 3272*x^9 + 7256*x^8 + 23624*x^7 + 1589326*x^6 - 300948*x^5 + 8153482*x^4 - 1027872*x^3 + 14644628*x^2 - 204356*x + 8784697}} \\[2mm]

12 & $11, {\begin{smallmatrix}5&0\\0&5\end{smallmatrix}}$ & \texttt{\detokenize{x^16 + 4*x^15 + 19*x^14 + 38*x^13 + 189*x^12 + 316*x^11 + 1376*x^10 + 2087*x^9 + 7661*x^8 + 8435*x^7 + 20625*x^6 + 8220*x^5 + 32165*x^4 + 4950*x^3 + 47175*x^2 - 775*x + 24025}} \\[2mm]
\end{longtable}

\section{Full Tabulation for imaginary \(D_{32}\) Fields with class number one }
\label{appendix:D32}

This section provides the complete list of imaginary, non-CM \(D_{32}\)-fields with class number one for which the extension over the unique degree-\(16\) normal subfield is ramified or unramified.

\subsection{Imaginary Non-CM \(D_{32}\) Fields Ramified over the \(D_{16}\) Subfield}

Let \(L/\mathbb{Q}\) denote an imaginary \(D_{32}\)-extension with class number one, and let \(N\) be its unique imaginary quadratic subfield such that the relative cyclic extension \(L/N\) is ramified at some finite places. For each admissible instance, the base field is of the form \(N=\mathbb{Q}(\sqrt{-d})\), and the ramification data is encoded by a nontrivial conductor \(\mathcal{F}_{L/N}\neq 1\).

\begin{longtable}{|c|c|M|}
\caption{Imaginary Non-CM \(D_{32}\) Fields with Class Number One which are Ramified over the \(N\).}\label{tab:D32}\\
\hline
\textbf{Entry} & \textbf{\([d, \mathfrak{F}_{L/N}]\)} & \textbf{Minimal Polynomial of \(L\)} \\
\hline
\endfirsthead
\hline
\textbf{Entry} & \textbf{\([d, \mathfrak{F}_{L/N}]\)} & \textbf{Minimal Polynomial of \(L\)} \\
\hline
\endhead
\hline
\endfoot

1 & $[2, {\begin{smallmatrix}16&0\\0&16\end{smallmatrix}}]$ & \texttt{\detokenize{x^32 - 64*x^29 + 128*x^28 + 192*x^27 + 304*x^26 - 2752*x^25 - 4204*x^24 + 28544*x^23 - 26464*x^22 - 77888*x^21 + 282896*x^20 - 227520*x^19 - 474352*x^18 + 1238592*x^17 - 1185890*x^16 + 461056*x^15 + 2209216*x^14 - 6405824*x^13 + 4902304*x^12 - 670912*x^11 - 84400*x^10 + 734144*x^9 + 178052*x^8 + 98176*x^7 + 179104*x^6 + 34624*x^5 + 14160*x^4 + 19648*x^3 + 4976*x^2 + 2240*x + 1225}} \\[2mm]

2 & $[3, {\begin{smallmatrix}64&0\\0&64\end{smallmatrix}}]$ & \texttt{\detokenize{x^32 + 240*x^30 + 59016*x^28 - 328960*x^26 + 3320772*x^24 + 12550752*x^22 + 17169216*x^20 + 30969216*x^18 + 48884562*x^16 + 16585376*x^14 - 12821280*x^12 - 2342784*x^10 + 12767992*x^8 + 9558912*x^6 + 4424544*x^4 + 1209600*x^2 + 202500}} \\[2mm]

3 & $[3, {\begin{smallmatrix}64&0\\0&64\end{smallmatrix}}]$ & \texttt{\detokenize{x^32 - 128*x^30 + 10896*x^28 - 519616*x^26 + 17791448*x^24 - 342558336*x^22 + 4845061696*x^20 - 49672825088*x^18 + 327894807816*x^16 - 415069485056*x^14 + 433066259968*x^12 - 27196075776*x^10 + 2289179456*x^8 - 231098368*x^6 + 8722176*x^4 - 204800*x^2 + 40000}} \\[2mm]

4 & $[19, {\begin{smallmatrix}64&0\\0&64\end{smallmatrix}}]$ & \texttt{\detokenize{x^32 - 4768*x^30 + 5726208*x^28 + 187486976*x^26 - 41675696*x^24 - 19152824064*x^22 + 1052772499328*x^20 - 17914209575552*x^18 + 104464325433456*x^16 - 36880547610368*x^14 + 122082854979584*x^12 - 55642824029184*x^10 + 41288727911296*x^8 - 11261408260096*x^6 + 2271287989248*x^4 - 119770921984*x^2 + 1929669184}} \\[2mm]

5 & $[19, {\begin{smallmatrix}64&0\\0&64\end{smallmatrix}}]$ & \texttt{\detokenize{x^32 + 4768*x^30 + 5726208*x^28 - 187486976*x^26 - 41675696*x^24 + 19152824064*x^22 + 1052772499328*x^20 + 17914209575552*x^18 + 104464325433456*x^16 + 36880547610368*x^14 + 122082854979584*x^12 + 55642824029184*x^10 + 41288727911296*x^8 + 11261408260096*x^6 + 2271287989248*x^4 + 119770921984*x^2 + 1929669184}} \\[2mm]

6 & $[43, {\begin{smallmatrix}64&0\\0&64\end{smallmatrix}}]$ & \texttt{\detokenize{x^32 + 191648*x^30 + 11366850512*x^28 + 5495439229568*x^26 + 4103118910788408*x^24 + 712637100771144448*x^22 + 269178104136110843904*x^20 - 11070213108637408578304*x^18 + 4521770494351319099708616*x^16 - 11614867697020421291689472*x^14 + 57977632335302621209344000*x^12 - 16670850928601537901023488*x^10 + 2646617952973719940278080*x^8 - 42961404896856351778816*x^6 + 350289020559841591552*x^4 - 13215368334744576*x^2 + 182233364544}} \\[2mm]

7 & $[43, {\begin{smallmatrix}64&0\\0&64\end{smallmatrix}}]$ & \texttt{\detokenize{x^32 - 191648*x^30 + 11366850512*x^28 - 5495439229568*x^26 + 4103118910788408*x^24 - 712637100771144448*x^22 + 269178104136110843904*x^20 + 11070213108637408578304*x^18 + 4521770494351319099708616*x^16 + 11614867697020421291689472*x^14 + 57977632335302621209344000*x^12 + 16670850928601537901023488*x^10 + 2646617952973719940278080*x^8 + 42961404896856351778816*x^6 + 350289020559841591552*x^4 + 13215368334744576*x^2 + 182233364544}} \\[2mm]

8 & $[67, {\begin{smallmatrix}64&0\\0&64\end{smallmatrix}}]$ & \texttt{\detokenize{x^32 - 53291696*x^30 + 753643643169032*x^28 - 1632855135589144992*x^26 + 1320382416547937836380*x^24 + 59962488127171450848352*x^22 + 13885949427147740586362272*x^20 + 4260712217162104473368852128*x^18 + 296494944370241405845034850874*x^16 + 5542833494316360637169845805344*x^14 + 81116869637795013022990100630272*x^12 - 91548894836778933930238528294400*x^10 + 43430222677273138880766523055304*x^8 + 6424843372469092034377974720*x^6 + 200071639976135254445696*x^4 - 2776870138893915200*x^2 + 8508597302500}} \\[2mm]

9 & $[67, {\begin{smallmatrix}64&0\\0&64\end{smallmatrix}}]$ & \texttt{\detokenize{x^32 + 53291696*x^30 + 753643643169032*x^28 + 1632855135589144992*x^26 + 1320382416547937836380*x^24 - 59962488127171450848352*x^22 + 13885949427147740586362272*x^20 - 4260712217162104473368852128*x^18 + 296494944370241405845034850874*x^16 - 5542833494316360637169845805344*x^14 + 81116869637795013022990100630272*x^12 + 91548894836778933930238528294400*x^10 + 43430222677273138880766523055304*x^8 - 6424843372469092034377974720*x^6 + 200071639976135254445696*x^4 + 2776870138893915200*x^2 + 8508597302500}} \\[2mm]

\end{longtable}

\subsection{Unramified Imaginary Non-CM \(D_{32}\) Fields with Class Number One}

A search over the admissible parameter ranges yields the complete set of positive integers \(d\) for which the imaginary quadratic field \(N=\mathbb{Q}(\sqrt{-d})\) admits an unramified non-CM \(D_{32}\)-extension with class number one.

\[
\{146,\ 407,\ 409,\ 471,\ 559,\ 1047,\ 1139,\ 1153,\ 1159,\ 1201,\ 1379,\ 1657,\ 1983,\ 2019,\ 2113,\ 2377,
\]
\[
2578,\ 2811,\ 2962,\ 3063,\ 3217,\ 3443,\ 3603,\ 4171,\ 4427,\ 4939,\ 4971,\ 7051,\ 7819,\ 8299,\ 8683,
\]
\[
9507,\ 10587,\ 10843,\ 12387,\ 12523,\ 13123,\ 14403,\ 15547,\ 16347,\ 16723,\ 25267,\ 28963\}.
\]

\appendix
\section{CM Fields with class number one of degrees 16 and 2-power degrees dihedral fields}
\label{AppendixA}

To contextualize the extensive computational enumeration of degree-\(16\) and higher \(2\)-power dihedral fields within this paper, we recall several definitive classification results for normal CM-fields of degree \(16\) with relative class number one. These theorems, drawn from the work of~\cite{lou7}, characterize all such fields for the specific Galois groups \(C_2 \times Q_8\), \(C_4 \circ D_8\) (denoted \(\mathbf{G}_9\) therein), and \(D_{16}\).

\begin{theorem}\cite[Theorem 6, $C_2\times Q_8$ CM fields]{lou7}
Let $\mathbf{N}$ be a non-abelian normal CM-field of degree $16$ which is the compositum of two normal octic CM-fields $\mathbf{N}_1$ and $\mathbf{N}_2$ with the same maximal real subfield $\mathbf{K}$. Assume that one of the $\mathbf{N}_i$ is a quaternion octic CM-field. Then, $h_{\mathbf{N}}^- = 1$ if and only if
\[
\mathbf{N} = \mathbf{Q}\left(\sqrt{-1}, \sqrt{(2 + \sqrt{2})(3 + \sqrt{3})}\right).
\]
Moreover, this number field has class number one with Galois group isomorphic to
\[
C_2\times Q_{8}=\{a,b,c\;:\; a^2=b^4=1, c^2=b^2, ab=ba, ac=ca, cbc^{-1}=b^{-1}\}.
\]
\end{theorem}

\begin{theorem}\cite[Theorem 20, $C_4\circ D_4$ CM fields]{lou7}
The number field
\[
\mathbf{N}
=\Q\!\left(\sqrt{2},\,\sqrt{5},\,\sqrt{37},\,\sqrt{-(2\sqrt{2}+3\sqrt{5})(2+\sqrt{5})}\right)
\]
is the only normal CM-field of degree $16$ with Galois group $\mathbf{G}_9$ of relative class number one. Moreover, this field has class number one with Galois group isomorphic to
\[
C_4\circ D_8=\{a,b,c\;:\; a^4=c^2=1, a^2=b^2, ab=ba, ac=ca, cbc=a^2b\}.
\]
\end{theorem}

\begin{theorem}\cite[Theorem 10 (b)]{lou7}
There are five dihedral CM-fields of degree $16$ with relative class numbers equal to one. Namely, the narrow Hilbert $2$-class fields $H_{n,2}(L)$ of the five real quadratic number fields $\Q(\sqrt{pq})$ with
\[
pq \in \{2\cdot 257,\ 5\cdot 101,\ 5\cdot 181,\ 13\cdot 53,\ 13\cdot 61\}.
\]
Finally, the narrow Hilbert $2$-class field of $L=\Q(\sqrt{2\cdot 257})$ has class number $3$, and the four remaining narrow Hilbert $2$-class fields have class number one.
\end{theorem}

\section{List of all imaginary non-CM $D_8$ fields, imaginary $C_2\times C_4$, and $C_2^3$ fields with class number one}\label{sec:deg_8}

\begin{theorem}[\cite{Yamamura1}]\label{thm:c2xc4_base}
The complete list of imaginary fields with Galois group \(C_2 \times C_4\) and class number one is:
\[
\begin{aligned}
&\mathbb{Q}(\zeta_{16}),\quad
\mathbb{Q}(\sqrt{-3}, \cos(\pi/8)),\quad
\mathbb{Q}(\zeta_{15}),\quad
\mathbb{Q}(\zeta_{20}),\\
&\mathbb{Q}(\sqrt{-7}, \zeta_5),\quad
\mathbb{Q}(\sqrt{-2}, \zeta_5),\quad
\mathbb{Q}(\sqrt{-3}, i\sin(\pi/8)),\quad
\mathbb{Q}(\sqrt{-11}, i\sin(\pi/8)),\\
&\mathbb{Q}(\zeta_5, \sqrt{2}),\quad
\mathbb{Q}(\zeta_5, \sqrt{13}),\quad
\mathbb{Q}(i\sin(\pi/8), \sqrt{5}),\quad
\mathbb{Q}(\zeta_5, \sqrt{17}).
\end{aligned}
\]
\end{theorem}

\begin{theorem}[\cite{Yamamura1}]\label{thm:c2cubed_base}
The complete list of imaginary fields with Galois group \(C_2^3\) and class number one is:
\[
\begin{aligned}
&\mathbb{Q}(\zeta_8,\sqrt{5}),\quad
\mathbb{Q}(\zeta_{24}),\quad
\mathbb{Q}(\zeta_{12},\sqrt{5}),\quad
\mathbb{Q}(\zeta_8,\sqrt{-11}),\\
&\mathbb{Q}(\sqrt{-3},\sqrt{5},\sqrt{-7}),\quad
\mathbb{Q}(\sqrt{-3},\sqrt{5},\sqrt{-2}),\quad
\mathbb{Q}(\sqrt{-1},\sqrt{5},\sqrt{-7}),\\
&\mathbb{Q}(\sqrt{-3},\sqrt{2},\sqrt{-11}),\quad
\mathbb{Q}(\sqrt{-7},\sqrt{5},\sqrt{-2}),\quad
\mathbb{Q}(\sqrt{-1},\sqrt{13},\sqrt{-7}),\\
&\mathbb{Q}(\sqrt{-3},\sqrt{17},\sqrt{-11}),\quad
\mathbb{Q}(\zeta_{12},\sqrt{-7}),\quad
\mathbb{Q}(\zeta_{12},\sqrt{-11}),\\
&\mathbb{Q}(\sqrt{-3},\sqrt{-7},\sqrt{-2}),\quad
\mathbb{Q}(\zeta_{12},\sqrt{-19}),\quad
\mathbb{Q}(\sqrt{-1},\sqrt{-7},\sqrt{-19}),\\
&\mathbb{Q}(\sqrt{-3},\sqrt{-11},\sqrt{-19}).
\end{aligned}
\]
\end{theorem}

\begin{theorem}[\cite{Y}]\label{thm:d8_base}
The complete list of imaginary non-CM fields with Galois group \(D_8\) and class number one is given in Table~\ref{tab:D8}. In the table, each entry provides the base field \(N=\mathbb{Q}(\sqrt{-d})\) (encoded by the pair \([d,\mathfrak{F}_{L/N}]\)) alongside the defining polynomial of \(L\).
\end{theorem}

\begin{longtable}{|c|p{10cm}|}
\caption{Imaginary non-CM \(D_8\) fields with class number one}\label{tab:D8}\\
\hline
\([d,\mathfrak{F}_{L/N}]\) & \textbf{Polynomial} \\
\hline
\endfirsthead
\hline
\([d,\mathfrak{F}_{L/N}]\) & \textbf{Polynomial} \\
\hline
\endhead
\hline
\endfoot

\([2, 4]\) & \(x^8 + 16x^6 + 100x^4 + 256x^2 + 196\) \\
\([2, 7]\) & \(x^8 + 8x^6 - 368x^4 + 2048x^2 + 65536\) \\
\([3, 16]\) & \(x^8 + 4x^6 + 18x^4 - 8x^2 + 4\) \\
\([3, 11]\) & \(x^8 - 66x^6 - 264x^5 + 1067x^4 + 11616x^3 + 39930x^2 + 66792x + 48037\) \\
\([7, 3]\) & \(x^8 + 3x^6 + 12x^4 - 9x^2 + 9\) \\
\([7, 19]\) & \(x^8 + 218x^6 + 1187x^4 + 2626x^2 + 729\) \\
\([11, 16]\) & \(x^8 + 12x^6 + 90x^4 + 104x^2 + 36\) \\
\([11, 7]\) & \(x^8 + 58x^6 + 187x^4 - 12806x^2 + 83521\) \\
\([11, 19]\) & \(x^8 + 2780x^6 + 2083126x^4 + 217282940x^2 + 6132612721\) \\
\([19, 16]\) & \(x^8 + 108x^6 + 3898x^4 + 47556x^2 + 13689\) \\
\([19, 3]\) & \(x^8 - 3x^6 + 21x^4 + 36x^2 + 144\) \\
\([19, 67]\) & \(x^8 + 76x^6 + 88998x^4 + 3327052x^2 + 9247681\) \\
\([43, 16]\) & \(x^8 + 20x^6 + 8546x^4 + 6200x^2 + 1156\) \\
\([43, 3]\) & \(x^8 + 318x^6 + 101127x^4 - 954x^2 + 9\) \\
\([43, 7]\) & \(x^8 + 186x^6 + 8203x^4 + 175242x^2 + 4774225\) \\
\([43, 163]\) & \(x^8 + 172x^6 + 504006x^4 + 42708460x^2 + 576816289\) \\
\([67, 16]\) & \(x^8 + 140x^6 + 12986x^4 + 2056100x^2 + 109725625\) \\
\([67, 3]\) & \(x^8 - 45x^6 + 744x^4 - 1053x^2 + 3249\) \\
\([67, 11]\) & \(x^8 + 378x^6 - 572x^5 + 38319x^4 + 45188x^3 + 852878x^2 + 5312164x + 7167733\) \\
\([67, 43]\) & \(x^8 + 2246x^6 + 467297x^4 - 1866632x^2 + 26010000\) \\
\([1, 8]\) & \(x^8 - 4x^6 + 8x^4 - 4x^2 + 1\) \\
\([1, 6]\) & \(x^8 - 14x^6 + 111x^4 + 46x^2 + 49\) \\
\([1, 12]\) & \(x^8 + 3x^4 + 9\) \\
\([1, 10]\) & \(x^8 - 4x^7 + 12x^6 - 26x^5 + 51x^4 - 50x^3 - 30x^2 + 50x + 25\) \\
\([1, 20]\) & \(x^8 - 2x^6 - 4x^5 + 25x^4 + 24x^3 - 80x^2 + 12x + 149\) \\
\([1, 7]\) & \(x^8 - 24x^6 + 118x^4 - 360x^2 + 9801\) \\
\([1, 22]\) & \(x^8 - 40x^6 + 446x^4 - 2680x^2 + 39601\) \\
\([1, 44]\) & \(x^8 - 40x^6 + 88x^5 + 666x^4 - 2112x^3 - 3208x^2 + 7480x + 19625\) \\
\([1, 26]\) & \(x^8 - 4x^7 + 20x^6 - 58x^5 + 159x^4 - 138x^3 - 46x^2 + 390x + 225\) \\
\([1, 52]\) & \(x^8 + 14x^6 - 36x^5 + 121x^4 + 216x^3 - 576x^2 + 108x + 2349\) \\
\([1, 38]\) & \(x^8 - 72x^6 + 1678x^4 - 19224x^2 + 245025\) \\
\([1, 76]\) & \(x^8 - 72x^6 + 152x^5 + 2058x^4 - 6080x^3 - 21352x^2 + 45752x + 150937\) \\
\([1, 74]\) & \(x^8 - 4x^7 + 36x^6 - 98x^5 + 611x^4 - 970x^3 + 2846x^2 - 1290x + 19325\) \\
\([1, 148]\) & \(x^8 + 14x^6 - 4x^5 + 513x^4 + 120x^3 + 2216x^2 + 1292x + 53861\) \\
\([1, 86]\) & \(x^8 - 168x^6 + 9982x^4 - 274680x^2 + 4626801\) \\
\([1, 134]\) & \(x^8 - 264x^6 + 25198x^4 - 1096920x^2 + 24591681\) \\
\([1, 326]\) & \(x^8 - 648x^6 + 155182x^4 - 16689240x^2 + 767899521\) \\
\([1, 652]\) & \(x^8 - 648x^6 - 1304x^5 + 158442x^4 + 427712x^3 - 16895272x^2 - 34004408x + 715078393\) \\
\([39, 1]\) & \(x^8 - x^6 + 4x^4 + 3x^2 + 9\) \\
\([55, 1]\) & \(x^8 + 5x^6 + 17x^4 + 20x^2 + 16\) \\
\([56, 1]\) & \(x^8 + 218x^6 + 18169x^4 + 704712x^2 + 11343424\) \\
\([68, 1]\) & \(x^8 + 292x^6 + 30004x^4 + 1225200x^2 + 14440000\) \\
\([136, 1]\) & \(x^8 + 546x^6 + 112081x^4 + 10086456x^2 + 332405824\) \\
\([203, 1]\) & \(x^8 + 842x^6 + 251203x^4 + 31891538x^2 + 1490346025\) \\
\([219, 1]\) & \(x^8 - 34x^6 + 1159x^4 + 102x^2 + 9\) \\
\([259, 1]\) & \(x^8 + 1058x^6 + 405323x^4 + 67341386x^2 + 4200724969\) \\
\([291, 1]\) & \(x^8 + 1190x^6 + 512281x^4 + 95857104x^2 + 6635079936\) \\
\([292, 1]\) & \(x^8 + 1204x^6 + 518868x^4 + 96089328x^2 + 6491202624\) \\
\([323, 1]\) & \(x^8 + 34x^6 + 611x^4 + 306x^2 + 81\) \\
\([328, 1]\) & \(x^8 + 1282x^6 + 641873x^4 + 145631128x^2 + 12445633600\) \\
\([388, 1]\) & \(x^8 + 1574x^6 + 909569x^4 + 229871500x^2 + 21506809104\) \\
\([723, 1]\) & \(x^8 + 2780x^6 + 3024214x^4 + 1557092028x^2 + 323980671249\) \\
\([763, 1]\) & \(x^8 + 3098x^6 + 3519443x^4 + 1746274082x^2 + 320497515625\) \\
\([772, 1]\) & \(x^8 + 3078x^6 + 3572785x^4 + 1849033692x^2 + 358853713936\) \\
\([1227, 1]\) & \(x^8 - 1475282x^6 + 581677732287x^4 - 156530114x^2 + 10609\) \\
\([1243, 1]\) & \(x^8 + 5100x^6 + 9492454x^4 + 7550399340x^2 + 2117609950401\) \\
\([1387, 1]\) & \(x^8 + 5702x^6 + 11678009x^4 + 10321828432x^2 + 3350305587456\) \\
\([1507, 1]\) & \(x^8 + 2350x^6 + 14381921x^4 + 33067320880x^2 + 21010709057536\) \\
\end{longtable}

\end{document}